\documentclass[11pt]{article}
\usepackage[margin=1in]{geometry}
\usepackage{amsmath,amssymb,amsthm,mathtools}

\usepackage[show]{ed}
\usepackage{graphicx}
\usepackage{booktabs}
\usepackage{enumitem}
\usepackage{hyperref}
\usepackage{setspace}
\usepackage{float}
\usepackage{bm}
\usepackage[nameinlink,noabbrev]{cleveref}
\usepackage{booktabs}
\usepackage{tabularx}
\usepackage{array}

\newtheorem{theorem}{Theorem}[section]
\newtheorem{proposition}[theorem]{Proposition}
\newtheorem{corollary}[theorem]{Corollary}
\newtheorem{assumption}[theorem]{Assumption}
\newtheorem{remark}[theorem]{Remark}

\begin{document}

\title{Ghost Dynamics in Receptor Signalling Networks:\\
A Fast--Slow Adaptive Extension of Competitive Cancer Inhibition Models}
\author{%
  G. Manjunath, Roumen Anguelov\\[6pt]
  Department of Mathematics and Applied Mathematics,
  University of Pretoria\\
  \{manjunath.gandhi,roumen.anguelov\}@up.ac.za
}
\date{}
\maketitle

\begin{abstract}
Receptor occupancy models quantify cancer-signalling inhibition but often equate target occupancy with downstream activity. This low-dimensional approximation cannot represent delayed pathway shutdown, transient resistance, or non-monotone viability responses generated by intracellular networks.

We formulate a fast--slow framework separating drug--target occupancy from downstream signalling. Fast $X$ represents occupancy, $A$ pro-survival activity, and slow $B$ adaptive feedback, including phosphatase induction, stress adaptation, and signalling rewiring. Rapid occupancy relaxation permits a quasi-steady reduction, while weak feedback guarantees global convergence to a unique equilibrium. Stronger feedback can bring the frozen activity subsystem near a saddle-node fold, producing ghost dynamics around a vanished high-activity equilibrium. The analysis also distinguishes frozen-subsystem folds from saddle-node equilibria of the coupled system and identifies the crossover ratio $\varepsilon\gamma_c/\delta^{3/2}$ separating the two asymptotic passage regimes relevant here. We determine when adaptive passage preserves the inverse-square-root residence-time law and when transverse crossing instead produces the Airy delay scale $O((\varepsilon\gamma_c)^{-1/3})$. For a concrete sigmoidal activity map and rational adaptive law, explicit parameter conditions produce a high-activity fold and transverse crossing. A directional boundary-layer estimate then transfers the reduced delay laws to the full occupancy system.

Coupling activity to viability converts signalling delays into dose--response shoulders, while adaptive lag provides a possible mechanism for early-time overshoot without ad hoc forcing. A proof-of-concept fit to limited time-resolved viability data is compatible with concentration-dependent delay across the measured conditions but exposes substantial practical-identifiability limitations. The framework predicts the apparent half-maximal inhibitory concentration $IC_{50}(t)$ because viability depends on integrated signalling history rather than receptor occupancy alone.
\end{abstract}



\section{Introduction}

Quantitative models of signalling inhibition are useful because they translate biochemical knowledge
into testable statements about therapeutic response. In receptor-mediated cancer signalling, a
common modelling strategy is to start from ligand--receptor binding kinetics and to use receptor
occupancy as a proxy for signalling activity \cite{Finlay2020}. This approach is attractive because the reactions are
mechanistically clear, the resulting differential equations are low-dimensional, and the parameters
can be related to measurable attachment and detachment rates.

A competitive activation--inhibition model for the CXCL12/CXCR4 axis adopts exactly this
philosophy \cite{Anguelov2023}. Activating ligand and inhibitor compete for a finite receptor pool, yielding a
low-dimensional mass-action system for receptor occupancy. The mathematical structure is clean:
the feasible region is positively invariant, the system admits a unique equilibrium, and all biologically
relevant trajectories converge to that equilibrium.  However, the same simplification that gives the model its tractability also
limits the range of signalling phenomena it can represent. In particular,
receptor occupancy and receptor activation are distinct pharmacological
concepts: occupancy of a receptor by a ligand does not in general determine
the magnitude or duration of receptor activation, and receptor activation
may persist after the activating ligand has dissociated
\cite{BlackLeff1983,Ritter2024}. This provides a natural motivation for separating the
state of receptor occupancy from the downstream functional response. In
addition, receptor activity is shaped by kinase--phosphatase feedback,
vesicular trafficking, endosomal signalling, and pathway rewiring within
the embedded biochemical network.

This distinction matters for therapy.  A treatment may rapidly shift receptor
occupancy toward inhibition while receptor activation and intracellular
circuitry temporarily sustain downstream pro-survival signalling. Thus, the
timescale of the pharmacological occupancy response need not coincide with
the timescale of the functional signalling response. Experimentally, this
can appear as delayed pathway shutdown, transient resistance, non-monotone
early-time viability curves, and time-dependent shifts in inferred potency
measures such as $IC_{50}(t)$.  Endocytosis and endosomal signalling are now recognized as integral components of
signal processing rather than mere signal termination mechanisms \cite{Sorkin2009,Miaczynska2013}.
Likewise, receptor-network criticality can generate transient memory through ghost dynamics near a
saddle-node threshold \cite{Stanoev2020}. 

Slow--fast and multiscale formulations have previously been used
in cancer modelling at several biological levels. Singular-perturbation
reductions have been applied to cancer-associated
epithelial--mesenchymal regulatory networks and to multiscale
epigenetic models of therapeutic response
\cite{AlRadhawiSontag2022,Alarcon2021}, while multiple-timescale
tumour--microenvironment models have been developed to connect
cellular and extracellular processes \cite{DuBois2013}. At the
receptor-network level, Stanoev et al.\ proposed a saddle-node ghost
as a mechanism for transient memory of time-varying growth-factor
signals \cite{Stanoev2020}. In cancer signalling, adaptive rewiring
and feedback relief can occur soon after targeted treatment
\cite{Chandarlapaty2012,Rosell2015,Pazarentzos2015}. The present
work connects these ideas by coupling fast drug--target occupancy
to slow adaptive signalling and by deriving conditions under which
frozen-ghost and dynamic-fold delay laws remain valid on the
passage intervals relevant to viability.

The purpose of this paper is to build a mathematically transparent bridge between the
competitive receptor-occupancy model and network-mediated signalling dynamics. Rather than
committing to a particular higher-dimensional trafficking model, we retain only the essential
property needed from the receptor layer: fast exponential relaxation of an occupancy state \cite{Jones1995,Kuehn2015}
$X(t;C)$ toward a quasi-steady state $X_{\mathrm{qs}}(C)$ determined by the applied inhibitor
concentration $C$. This makes the present work an extension of the original competitive model at
the level of dynamical principle, not by literal inclusion of all receptor variables.

The contribution is fourfold. First, we introduce a
network-mediated activity variable together with a slow adaptive
regulator. Second, we prove convergence in the weak-feedback regime,
establish fixed-time quasi-steady reduction, obtain uniform-in-time
reduction under contraction feedback, and derive a
forcing-weighted stability criterion for reduction on
fold-relevant diverging time intervals. Third, we characterize the
fold geometry of the frozen activity subsystem and distinguish this
fold from a saddle-node equilibrium of the coupled
activity--adaptation system. We identify the crossover ratio
\(\varepsilon\gamma_c/\delta^{3/2}\), recover the frozen ghost law
\(T\sim\pi/\sqrt{\delta}\) when adaptive drift is negligible, and,
under transverse adaptive passage, obtain the Airy delay law
\(t_{\mathrm{exit}}-t_c
\sim z_{\mathrm{Ai}}(\varepsilon\gamma_c)^{-1/3}\), together with
the exit displacement
\(\mu(t_{\mathrm{exit}})
\sim-z_{\mathrm{Ai}}(\varepsilon\gamma_c)^{2/3}\).
For the sigmoidal activity map and rational adaptive law, we give
explicit parameter conditions for a high-activity fold and transverse
adaptive passage. A directional boundary-layer estimate, exploiting the
fact that occupancy error initially forces only the activity component,
then transfers the reduced delay laws to the full system under the scale
conditions $\eta=o(\delta)$ and
$\eta=o((\varepsilon\gamma_c)^{2/3})$.
Fourth, we show how coupling
activity to viability produces dose--response shoulders and identify
adaptive lag as a possible mechanism for low-dose overshoot. A
proof-of-concept fit to time-resolved viability data illustrates a
concentration-dependent delay phenotype and its practical
identifiability limitations.

The classical saddle-node normal form, inverse-square-root ghost law, and Airy delay law \cite{Neishtadt1988} are not themselves new. The mathematical
contribution lies in connecting these local laws to a
receptor-occupancy reduction on parameter-dependent passage
intervals, distinguishing frozen-subsystem folds from coupled-system
saddle-node equilibria, and identifying conditions under which the
resulting delays persist in the full occupancy--signalling system.

The paper is organized as follows. Section~\ref{sec:occupancy} introduces the fast receptor-occupancy layer, and Section~\ref{sec:activity-adaptation} formulates the downstream
activity--adaptation extension. Section~\ref{sec:global-dynamics} establishes global convergence
under contraction feedback, while Section~\ref{sec:quasi-steady} justifies the quasi-steady
reduction of the occupancy variables. Section~\ref{sec:ghost} develops the local saddle-node
normal form and the corresponding ghost residence-time law. Section~\ref{sec:bifurcation}
characterizes the fold curve and local phase portrait of the frozen activity subsystem.
Section~\ref{sec:coupled-fold} reconnects this frozen-fold analysis to the coupled fast--slow
dynamics by distinguishing coupled equilibria from frozen activity folds and by analyzing slow
adaptive passage through the fold. Section~\ref{sec:viability} translates signalling histories
into viability predictions, including shoulder formation, adaptive overshoot, and
time-dependent potency, while Subsection~\ref{sec:Lkyn-fit} presents the proof-of-concept fit to
experimental viability data. Section~\ref{sec:discussion} concludes with a discussion of the
biological interpretation, limitations, and possible extensions.

\section{Fast receptor occupancy as an abstract input layer}\label{sec:occupancy}

The bio-theoretical setting of an occupancy-based inhibition model is a
competitive system in which two ligands---an activator and an
inhibitor---bind to the same site on a receptor. The reversible binding
reactions are represented by a dynamical system consisting of ordinary
differential equations. A representative example is the model presented
in \cite{Anguelov2023}, in which the occupancy state converges
exponentially to a unique equilibrium.  The main purpose of the model is to describe the time evolution of the occupancy state for a given concentration $C$ of the inhibitor. For the network-mediated activity extension of the model discussed in the sequel, the detailed biochemical form of the occupancy is not essential. In fact the model need not be limited to two equations. What is essential is the inherent timescale separation: the occupancy state must relax to equilibrium rapidly relative to the downstream signalling dynamics it influences. We use the term “receptor” in the broad pharmacological sense of a molecular target with which a drug interacts to produce its specific effect. In this usage, the target need not be restricted to a membrane receptor, but may also be an enzyme, ion channel, transporter, or other drug target \cite{Kenakin2019, Ritter2024}.
Consequently, the occupancy variable introduced below may be interpreted generically as a fast variable describing the state of the drug–target interaction.

We therefore introduce an abstract occupancy variable $X(t;C)\in\mathcal X$,
where $\mathcal X$ is a biologically feasible occupancy state space. The
variable $X$ may be one-dimensional, for example representing the effective
inhibitor-bound target fraction, or multi-dimensional, for example
representing several target-occupancy observables. The present analysis does
not require specifying this structure explicitly.

We assume that receptor occupancy evolves on a fast timescale according to
\begin{equation}
\dot X = \frac{1}{\eta}F_X(X;C),
\qquad 0<\eta\ll 1,
\label{eq:fastX}
\end{equation}
where $C$ denotes the externally applied inhibitor concentration. The small parameter $\eta$ expresses the fact that receptor occupancy
approaches its quasi-steady state exponentially faster than intracellular
signalling and adaptive feedback.

\begin{assumption}[Fast occupancy relaxation]
For each fixed inhibitor concentration $C$, the system \eqref{eq:fastX} admits a unique
quasi-steady state $X_{\mathrm{qs}}(C)\in\mathcal X$. Moreover, there exist constants
$K,\kappa>0$ such that every solution with initial condition in $\mathcal X$ satisfies
\begin{equation}
\|X(t;C)-X_{\mathrm{qs}}(C)\|
\le
K e^{-\kappa t/\eta}.
\label{eq:Xfastconv}
\end{equation}
\end{assumption}

This assumption is motivated by a competitive receptor-occupancy model, in which the
occupancy variables converge to a stable equilibrium and, after a fast transient, evolve close to a
slow manifold. In the present formulation, the model in \cite{Anguelov2023} serves as the motivating example for
the exponential relaxation hypothesis rather than as a subsystem that must be embedded in full
detail.

\section{Network-mediated activity and adaptive regulation}
\label{sec:activity-adaptation}
To distinguish receptor occupancy from signalling output, we introduce an activity variable
$A(t)$ and a slow adaptive regulator $B(t)$. The variable $A$ represents effective pro-survival
signalling output, while $B$ summarizes induced negative regulators, stress-adaptive rewiring,
or compensatory signalling mechanisms.

The adaptive variable $B$ acts at the signalling-network level rather than at the receptor-binding
level. Consequently, in the present formulation the receptor occupancy subsystem depends on the
external inhibitor concentration $C$ but not directly on $B$. The quasi-steady occupancy state is
therefore written as $X_{\mathrm{qs}}=X_{\mathrm{qs}}(C)$ rather than
$X_{\mathrm{qs}}(C,B)$. A dependence on $B$ would require additional biological
mechanisms, such as receptor downregulation, altered trafficking, or modified binding kinetics,
to be included explicitly in the occupancy equations.

After the fast occupancy variable has relaxed, the signalling system receives the quasi-steady
occupancy input $X_{\mathrm{qs}}(C)$. We therefore consider
\begin{align}
\dot A
&=
\lambda\Big(\Phi(X_{\mathrm{qs}}(C),A,B)-A\Big),
\qquad \lambda>0,
\label{eq:A}\\
\dot B
&=
\varepsilon\big(\Psi(A,C)-B\big),
\qquad 0<\varepsilon\ll 1.
\label{eq:B}
\end{align}

Here $\Phi:\mathcal X\times[0,1]^2\to[0,1]$ encodes network-mediated
signalling activity \cite{Stanoev2020,Aldridge2006}, while
$\Psi:[0,1]\times[0,\infty)\to[0,1]$ encodes slow adaptive induction.
For the analytical prototype below, concentration is
nondimensionalized by a fixed reference concentration and the same symbol $C$ is retained.
Define the net occupancy-mediated drive
\begin{equation}
h(C)
=
\alpha X_{\mathrm{act}}(C)
-\beta X_{\mathrm{inh}}(C)
+\eta_0 .
\label{eq:net-occupancy-drive}
\end{equation}
We assume $h\in C^3$ and $h'(C)<0$ on the concentration interval of
interest. Thus increasing inhibitor concentration decreases the direct
signalling drive. A useful prototype is
\begin{equation}
\Phi
=
\sigma\big(
h(C)+\gamma A-\rho B
\big),
\label{eq:Phi}
\end{equation}
where $\sigma(z)=1/(1+e^{-z})$ is a sigmoid. Here $X_{\mathrm{act}}(C)$ and
$X_{\mathrm{inh}}(C)$ denote effective activating and inhibiting components extracted from
$X_{\mathrm{qs}}(C)$. In a one-dimensional reduction, these may be replaced by a single
effective occupancy input.

A simple adaptive law is
\begin{equation}
\Psi(A,C)
=
\frac{\kappa_A A+\kappa_C C}{1+\kappa_A A+\kappa_C C},
\label{eq:Psi}
\end{equation}
although any smooth increasing function may be used.

This architecture is motivated by three biological observations. First, receptor occupancy is not
itself identical to signalling activity; downstream signalling can persist after receptor binding has
changed. Second, feedback architecture strongly shapes pathway amplitude and timing. Third,
adaptive resistance to targeted therapy can emerge through signalling rewiring and feedback relief.
The variable $B$ is therefore not an occupancy variable. It is an effective slow state of the
signalling network. Consequently, in the present formulation $B$ affects signalling activity through
$\Phi$ and adaptation through $\Psi$, but it does not directly alter the fast occupancy equilibrium
$X_{\mathrm{qs}}(C)$.  
Table~\ref{tab:variables-observables} summarizes this hierarchy and
foreshadows the connection between the model states and experimentally
accessible quantities, including the downstream viable-cell variable
\(N\) introduced in Section~\ref{sec:viability}.

\begin{table}[tbp]
\centering
\caption{Variables of the occupancy--signalling--adaptation system,
their biological roles, characteristic timescales, and candidate
experimental observables. The variables \(A\) and \(B\) are effective
network states, so the listed measurements are possible proxies rather
than one-to-one identifications.}
\label{tab:variables-observables}
{\small
\setlength{\tabcolsep}{3.5pt}
\renewcommand{\arraystretch}{1.18}
\begin{tabularx}{\linewidth}{@{}
    >{\centering\arraybackslash}p{0.08\linewidth}
    >{\raggedright\arraybackslash}p{0.25\linewidth}
    >{\raggedright\arraybackslash}p{0.18\linewidth}
    >{\raggedright\arraybackslash}X
    @{}}
\toprule
\textbf{Variable}
&
\textbf{Biological role}
&
\textbf{Timescale}
&
\textbf{Candidate observable}
\\
\midrule

\(X\)
&
Drug--target occupancy state, possibly vector-valued
&
Fast:
\(t_X=O(\eta)\)
&
Bound-target fraction or competitive-binding/occupancy assay
\\[2pt]

\(A\)
&
Effective pro-survival signalling activity
&
Intermediate:
\(t_A=O(\lambda^{-1})\)
&
Time-resolved pathway reporter or phospho-signalling measurement,
such as pERK or pAKT
\\[2pt]

\(B\)
&
Adaptive feedback, stress response, or pathway rewiring
&
Slow:
\(t_B=O(\varepsilon^{-1})\)
&
Phosphatase or feedback-protein abundance; stress-response,
transcriptomic, or proteomic rewiring markers
\\

\bottomrule
\end{tabularx}
}
\end{table}

\section{Global dynamics in the weak-feedback regime}
\label{sec:global-dynamics}
We first show that the reduced signalling--adaptation system preserves asymptotic simplicity
whenever intracellular feedback is sufficiently weak. The fast occupancy layer has already been
absorbed into the quasi-steady input $X_{\mathrm{qs}}(C)$.

\begin{theorem}[Global convergence under contraction feedback]
Assume:
\begin{enumerate}[label=(A\arabic*)]
\item for each fixed $C$, the fast occupancy subsystem admits a unique exponentially stable
quasi-steady state $X_{\mathrm{qs}}(C)$;
\item $\Phi$ and $\Psi$ are continuous and map their domains into $[0,1]$;
\item for each fixed \(C\), the map
\(\mathcal T_C(A,B)=\bigl(\Phi(X_{\mathrm{qs}}(C),A,B),\Psi(A,C)\bigr)\)
is a contraction in the Euclidean norm: there exists \(q<1\) such that
\[
\|\mathcal T_C(z)-\mathcal T_C(w)\|_2
\le q\|z-w\|_2
\qquad
\text{for all }z,w\in[0,1]^2.
\]
\end{enumerate}
Then, for each fixed \(C\), the map \(\mathcal T_C\) has a unique fixed
point \(z^*(C)=\bigl(A^*(C),B^*(C)\bigr)\in[0,1]^2\), characterized by
\[
A^*(C)
=
\Phi\bigl(X_{\mathrm{qs}}(C),A^*(C),B^*(C)\bigr),
\qquad
B^*(C)
=
\Psi\bigl(A^*(C),C\bigr).
\]
This fixed point is the unique equilibrium of the reduced
activity--adaptation system on \([0,1]^2\), and it is globally
exponentially asymptotically stable.
\end{theorem}

\begin{proof}
Because \(\mathcal T_C\) maps \([0,1]^2\) into itself and is a contraction,
the Banach fixed-point theorem gives a unique fixed point
\(z^*=(A^*,B^*)\). This fixed point is precisely the unique equilibrium of
\eqref{eq:A}--\eqref{eq:B}. Write
\[
z=
\begin{pmatrix}A\\B\end{pmatrix},
\qquad
D=
\begin{pmatrix}
\lambda&0\\
0&\varepsilon
\end{pmatrix}.
\]
The reduced system is $\dot z=D\bigl(\mathcal T_C(z)-z\bigr)$.
Consider \(W(z)=\frac12(z-z^*)^\mathsf{T}D^{-1}(z-z^*)\).
Consequently,
\[
\dot W
\le
-(1-q)\|z-z^*\|_2^2
\le
-2(1-q)\min\{\lambda,\varepsilon\}\,W.
\]
Hence \(W(z(t))\le W(z(0))
\exp\!\left(-2(1-q)\min\{\lambda,\varepsilon\}\,t\right)\).
Therefore every trajectory in the positively invariant square
\([0,1]^2\) converges exponentially to \(z^*\), proving that \(z^*\) is
globally exponentially asymptotically stable.
\end{proof}

\begin{remark} \rm
This theorem shows that adding signalling and adaptation does not by itself create new long-term
states. Rich transient behaviour appears when the activity--adaptation subsystem is close to
losing the contraction property, for example near a saddle-node threshold.
\end{remark}

\section{Justification of the quasi-steady reduction}
\label{sec:quasi-steady}

The reduction from the full occupancy--signalling system \cite{Jones1995,Kuehn2015} to \eqref{eq:A}--\eqref{eq:B}
is justified by the exponential relaxation assumption on $X$. The unreduced system has the form
\begin{align}
\dot X
&=
\frac{1}{\eta}F_X(X;C),
\qquad 0<\eta\ll1,
\label{eq:fullX}\\
\dot A
&=
\lambda\big(\Phi(X,A,B)-A\big),
\label{eq:fullA}\\
\dot B
&=
\varepsilon\big(\Psi(A,C)-B\big),
\qquad 0<\varepsilon\ll1.
\label{eq:fullB}
\end{align}

The key point is that the fast subsystem depends on $C$ but not on $B$. Thus the quasi-steady
occupancy state is $X_{\mathrm{qs}}(C)$. If one wanted to use
$X_{\mathrm{qs}}(C,B)$ instead, then $B$ would have to appear explicitly in the occupancy equation,
for example through receptor availability, trafficking rates, or binding constants. Since such a
feedback is not included here, we do not assume it.

Replacing the fast occupancy state $X^\eta(t)$ by its
quasi-steady value $X_{\mathrm{qs}}(C)$ gives the reduced
downstream system
\begin{align}
\dot{\bar A}
&=
\lambda\Bigl(
\Phi\bigl(X_{\mathrm{qs}}(C),\bar A,\bar B\bigr)-\bar A
\Bigr),
\label{eq:reducedA}\\
\dot{\bar B}
&=
\varepsilon\bigl(\Psi(\bar A,C)-\bar B\bigr).
\label{eq:reducedB}
\end{align}
We now compare solutions of the full system
\eqref{eq:fullX}--\eqref{eq:fullB} with solutions of the
reduced system \eqref{eq:reducedA}--\eqref{eq:reducedB}.

\begin{theorem}[Fixed-time validity and a long-time criterion for
the quasi-steady reduction]
\label{thm:quasi-steady-validity}
Fix $C$. For each $\eta>0$, let $(X^\eta,A^\eta,B^\eta)$ solve
\eqref{eq:fullX}--\eqref{eq:fullB}, and let $(\bar A,\bar B)$
solve \eqref{eq:reducedA}--\eqref{eq:reducedB} with
$\bar A(0)=A^\eta(0)$ and $\bar B(0)=B^\eta(0)$.

Assume that \eqref{eq:Xfastconv} holds and that, on the relevant
invariant set, there are constants $L_X,L_S,L_\Psi>0$ such that
\[
\begin{aligned}
|\Phi(X_1,A_1,B_1)-\Phi(X_2,A_2,B_2)|
&\leq L_X\|X_1-X_2\|
 +L_S\bigl(|A_1-A_2|+|B_1-B_2|\bigr),\\
|\Psi(A_1,C)-\Psi(A_2,C)|
&\leq L_\Psi|A_1-A_2|.
\end{aligned}
\]
Set $L_0=\lambda(L_S+1)+\varepsilon(L_\Psi+1)$.

\begin{enumerate}[label=\textup{(\roman*)}]
\item
For every fixed $T>0$, there exists $C_T>0$, independent of
$\eta$, such that
\[
\sup_{0\leq t\leq T}
\bigl(
|A^\eta(t)-\bar A(t)|
+
|B^\eta(t)-\bar B(t)|
\bigr)
\leq C_T\eta.
\]
More precisely, for $0\leq t\leq T$,
\begin{equation}
|A^\eta(t)-\bar A(t)|
+
|B^\eta(t)-\bar B(t)|
\leq
\lambda L_XK
\frac{e^{L_0t}-e^{-\kappa t/\eta}}
     {L_0+\kappa/\eta}.
\label{eq:fixed-time-reduction}
\end{equation}

\item
Let \(T_\eta>0\) depend on \(\eta\) and possibly on the fold
parameters, with \(T_\eta\to\infty\) allowed as \(\eta\downarrow0\).
Write \(z=(A,B)^{\mathsf T}\), set
\(\bar z=(\bar A,\bar B)^{\mathsf T}\), and define
\[
G_C(A,B)
=
\bigl(
\lambda[\Phi(X_{\mathrm{qs}}(C),A,B)-A],
\varepsilon[\Psi(A,C)-B]
\bigr)^{\mathsf T}.
\]

Fix \(r_\eta>0\) and, for \(0\leq t\leq T_\eta\), define
\(\mathcal N_{\eta,r}(t)=\{z:\|z-\bar z(t)\|_2<r_\eta\}\).
Assume that these tubes lie in the domain of the downstream
vector field and that there is an integrable function
\(\ell_\eta:[0,T_\eta]\to\mathbb R\) such that
\[
\langle z_1-z_2,G_C(z_1)-G_C(z_2)\rangle
\leq
\ell_\eta(t)\|z_1-z_2\|_2^2
\]
for all \(z_1,z_2\in\mathcal N_{\eta,r}(t)\).

Define the forcing-weighted amplification factor by
\begin{equation}
\mathcal A_\eta(T_\eta)
=
\sup_{0\leq t\leq T_\eta}
\frac{1}{\eta}
\int_0^t
\exp\left(\int_s^t\ell_\eta(r)\,dr\right)
e^{-\kappa s/\eta}\,ds .
\label{eq:weighted-amplification}
\end{equation}
If \(Q_\eta:=\lambda L_XK\,\eta\mathcal A_\eta(T_\eta)<r_\eta\),
then the full downstream trajectory remains in
\(\mathcal N_{\eta,r}(t)\) for \(0\leq t\leq T_\eta\), and
\begin{equation}
\sup_{0\leq t\leq T_\eta}
\|z^\eta(t)-\bar z(t)\|_2
\leq
\lambda L_XK\,\eta\mathcal A_\eta(T_\eta).
\label{eq:long-time-reduction}
\end{equation}

Consequently, the quasi-steady reduction is valid
on \([0,T_\eta]\) whenever
\(\eta\mathcal A_\eta(T_\eta)\to0\) and
\(Q_\eta<r_\eta\) for all sufficiently small \(\eta\).
A convenient sufficient condition is
\(\eta\mathcal A_\eta(T_\eta)=o(r_\eta)\).
If \(\mathcal A_\eta(T_\eta)\) is uniformly bounded and
\(\inf_\eta r_\eta>0\), then the approximation error is
\(O(\eta)\) throughout the interval.
\end{enumerate}
\end{theorem}

\begin{proof}
For the fixed-time statement, set
$E_\eta(t)=|A^\eta(t)-\bar A(t)|
+|B^\eta(t)-\bar B(t)|$. The Lipschitz assumptions and
\eqref{eq:Xfastconv} give
$D^+E_\eta(t)\leq\lambda L_XK e^{-\kappa t/\eta}
+L_0E_\eta(t)$. Since $E_\eta(0)=0$, Gronwall's inequality gives
\[
E_\eta(t)
\leq
\lambda L_XK
\int_0^t e^{L_0(t-s)}e^{-\kappa s/\eta}\,ds
=
\lambda L_XK
\frac{e^{L_0t}-e^{-\kappa t/\eta}}
     {L_0+\kappa/\eta}.
\]
In particular,
$E_\eta(t)\leq(\lambda L_XK/\kappa)e^{L_0T}\eta$ for
$0\leq t\leq T$, proving part~\textup{(i)}.

For part~\textup{(ii)}, write
\(z^\eta=(A^\eta,B^\eta)^{\mathsf T}\). The full downstream
equation is
\(\dot z^\eta=G_C(z^\eta)+p^\eta(t,z^\eta)\), where
\[
p^\eta(t,A,B)
=
\bigl(
\lambda[
\Phi(X^\eta(t),A,B)
-\Phi(X_{\mathrm{qs}}(C),A,B)
],
0
\bigr)^{\mathsf T}.
\]
By \eqref{eq:Xfastconv},
\(\|p^\eta(t,z)\|_2\leq\lambda L_XK e^{-\kappa t/\eta}\).

Set \(\mathcal E_\eta(t)=\|z^\eta(t)-\bar z(t)\|_2\) and define
\(\tau_\eta=\inf\{t\in[0,T_\eta]:\mathcal E_\eta(t)\geq r_\eta\}\),
with \(\inf\varnothing=T_\eta\). For \(0\leq t<\tau_\eta\),
both \(z^\eta(t)\) and \(\bar z(t)\) belong to
\(\mathcal N_{\eta,r}(t)\). Hence
\(D^+\mathcal E_\eta(t)\leq
\ell_\eta(t)\mathcal E_\eta(t)+\lambda L_XK e^{-\kappa t/\eta}\).
Since \(\mathcal E_\eta(0)=0\), Gronwall's inequality gives
\[
\mathcal E_\eta(t)
\leq
\lambda L_XK
\int_0^t
\exp\left(\int_s^t\ell_\eta(r)\,dr\right)
e^{-\kappa s/\eta}\,ds
\leq Q_\eta
\]
for \(0\leq t<\tau_\eta\). By continuity, the same estimate
holds at \(t=\tau_\eta\). Because \(Q_\eta<r_\eta\), an exit at
any time \(\tau_\eta<T_\eta\) is impossible. Thus
\(\tau_\eta=T_\eta\), and the preceding estimate proves
\eqref{eq:long-time-reduction}.
\end{proof}

\begin{remark}
Part~\textup{(i)} recovers the original fixed-time theorem.
Part~\textup{(ii)} remains applicable on diverging intervals
because it measures the amplification of the exponentially
localized occupancy error rather than bounding it by
$e^{L_0T_\eta}$.

For a frozen-ghost passage, the leading delay is preserved if
$\eta\mathcal A_\eta(T_\eta)=o(\delta)$. For a transverse
dynamic-fold passage, it is preserved if
$\eta\mathcal A_\eta(T_\eta)
=o((\varepsilon\gamma_c)^{2/3})$.
\end{remark}

\begin{corollary}[Uniform-in-time reduction under contraction
feedback]
\label{cor:uniform-reduction-contraction}
Assume the hypotheses of Theorem~4.1 and the occupancy and
Lipschitz assumptions of Theorem~\ref{thm:quasi-steady-validity}.
Then there is a constant $C>0$, independent of $\eta$ and time,
such that
\[
\sup_{t\geq0}
\left\|
\begin{pmatrix}
A^\eta(t)-\bar A(t)\\
B^\eta(t)-\bar B(t)
\end{pmatrix}
\right\|_2
\leq C\eta .
\]
Thus the quasi-steady reduction is uniformly valid for all
$t\geq0$ in the weak-feedback regime.
\end{corollary}

\begin{proof}
Let $D=\operatorname{diag}(\lambda,\varepsilon)$ and equip
$\mathbb R^2$ with the weighted norm
$\|z\|_{D^{-1}}=(z^{\mathsf T}D^{-1}z)^{1/2}$. The contraction
argument in Theorem~4.1 applies to the difference of any two
solutions of the reduced downstream system and gives
\[
\|\varphi_C(t,z_1)-\varphi_C(t,z_2)\|_{D^{-1}}
\leq
e^{-ct}\|z_1-z_2\|_{D^{-1}},
\qquad
c=(1-q)\min\{\lambda,\varepsilon\}.
\]

The perturbation produced by the unreduced occupancy variable satisfies
\(\|p^\eta(t,z)\|_{D^{-1}}
\leq\sqrt{\lambda}\,L_XK e^{-\kappa t/\eta}\).
Consequently,
\begin{align*}
\|z^\eta(t)-\bar z(t)\|_{D^{-1}}
&\leq
\sqrt{\lambda}\,L_XK
\int_0^t
e^{-c(t-s)}e^{-\kappa s/\eta}\,ds\\
&\leq
\frac{\sqrt{\lambda}\,L_XK}{\kappa}\eta .
\end{align*}
Equivalence of the weighted and Euclidean norms completes the
proof.
\end{proof}

\begin{corollary}[Explicit reduction bound under bounded accumulated growth]
\label{cor:explicit-reduction-bound}
Assume the hypotheses of
Theorem~\ref{thm:quasi-steady-validity}\textup{(ii)}. Suppose that,
on the comparison interval $[0,T_\eta]$, the one-sided growth
rate can be chosen so that there are constants
$\Lambda_*\geq0$ and $L_*>0$, independent of the joint small-parameter
family, with
\begin{equation}
\ell_\eta(t)\geq-L_*,
\qquad
\sup_{0\leq t\leq T_\eta}
\int_0^t\ell_\eta(r)\,dr
\leq\Lambda_*.
\label{eq:bounded-accumulated-growth}
\end{equation}
Then, whenever $\eta L_*<\kappa$,
\begin{equation}
\mathcal A_\eta(T_\eta)
\leq
\frac{e^{\Lambda_*}}{\kappa-\eta L_*}.
\label{eq:explicit-amplification-bound}
\end{equation}
Consequently, if
\(\frac{\lambda L_XK e^{\Lambda_*}}{\kappa-\eta L_*}\,\eta<r_\eta\),
then
\begin{equation}
\sup_{0\leq t\leq T_\eta}
\|z^\eta(t)-\bar z(t)\|_2
\leq
\frac{\lambda L_XK e^{\Lambda_*}}
     {\kappa-\eta L_*}\,\eta.
\label{eq:explicit-long-time-reduction}
\end{equation}
\end{corollary}

\begin{proof}
For $0\leq s\leq t\leq T_\eta$,
\[
\int_s^t\ell_\eta(r)\,dr
=
\int_0^t\ell_\eta(r)\,dr
-
\int_0^s\ell_\eta(r)\,dr
\leq
\Lambda_*+L_*s.
\]
Using this estimate in
\eqref{eq:weighted-amplification} gives
\[
\mathcal A_\eta(T_\eta)
\leq
\frac{e^{\Lambda_*}}{\eta}
\int_0^\infty
e^{-(\kappa/\eta-L_*)s}\,ds
=
\frac{e^{\Lambda_*}}{\kappa-\eta L_*}.
\]
The trajectory estimate follows from
Theorem~\ref{thm:quasi-steady-validity}\textup{(ii)}.
\end{proof}

Under \eqref{eq:bounded-accumulated-growth}, the reduction error
is $O(\eta)$ even on a diverging comparison interval. This general
criterion is not invoked for the transverse fold of the concrete
sigmoidal model: Subsection~\ref{sec:directional-transfer} shows why
an all-direction Euclidean growth estimate is too strong there and
uses a directional occupancy-forcing estimate instead.

\section{Reduced activity dynamics and ghost phenomena}
\label{sec:ghost}
Having reduced the fast occupancy layer to the quasi-steady input
$X_{\mathrm{qs}}(C)$, we now analyze the downstream activity--adaptation
system. The applied inhibitor concentration $C$ does not act directly on
the activity variable. Rather, it first modifies the quasi-steady receptor
occupancy $X_{\mathrm{qs}}(C)$.

To make the direction of this effect explicit, we introduce an effective
inhibition variable
\begin{equation}
I(C)=H\bigl(X_{\mathrm{qs}}(C)\bigr),
\label{eq:IofC}
\end{equation}
where $H$ extracts the inhibitory component of the quasi-steady occupancy
state. We assume that increasing inhibitor concentration increases effective
inhibition, so that
\begin{equation}
I'(C)>0.
\label{eq:Imonotone}
\end{equation}

The reduced activity equation may therefore be written as
\begin{equation}
\dot A
=
F(A,B;I)
:=
\lambda\Big(\widetilde{\Phi}(I,A,B)-A\Big),
\label{eq:FABI}
\end{equation}
where $\widetilde{\Phi}$ is the activity map expressed in terms of the
effective inhibition variable. We assume that increasing effective
inhibition suppresses signalling activity,
\begin{equation}
\partial_I F(A,B;I)<0.
\label{eq:FImonotone}
\end{equation}
Thus increasing $C$ increases $I(C)$ and shifts the activity vector field in
the inhibitory direction.

The adaptive variable \(B\) does not modify receptor occupancy directly.
Instead, \(B\) reshapes the signalling landscape through the activity map
\(\widetilde{\Phi}\). Thus the treatment concentration \(C\) enters the
reduced activity--adaptation system in two ways: through the
occupancy-induced inhibition \(I(C)=H\bigl(X_{\mathrm{qs}}(C)\bigr)\)
in the activity equation, and through the adaptive induction law
\(\Psi(A,C)\) in the adaptation equation.

A saddle-node point occurs at $(A_c,B_c,I_c)$ if
\begin{equation}
F(A_c,B_c;I_c)=0,\qquad
\partial_A F(A_c,B_c;I_c)=0,\qquad
\partial_{AA}F(A_c,B_c;I_c)\neq 0.
\label{eq:foldconds}
\end{equation}
and the unfolding with respect to \((B,I)\) is nondegenerate.
Near the fold, the unfolding parameter may be written locally as $\mu
=
\alpha(B-B_c)-\beta(I-I_c)
+\text{higher-order terms},
\qquad \beta>0.$
Hence increasing inhibition decreases $\mu$ and drives the system through
the saddle-node threshold.

\begin{proposition}[Local saddle-node reduction] \label{prop:local-saddle-node}
Assume $F\in C^3$ near $(A_c,B_c;I_c)$ and that \eqref{eq:foldconds} holds together with
\((\partial_BF,\partial_IF)(A_c,B_c;I_c)\neq (0,0)\).
Then, after the smooth shift \(a=A-A_c\), \(\mu=\mu(B,I)\),
and rescaling of $a$ and time, the reduced equation is locally equivalent to
\begin{equation}
\dot a = \mu-a^2+R(\mu,a),\qquad R(\mu,a)=O(|\mu||a|+|a|^3+|\mu|^2).
\label{eq:normalform}
\end{equation}
\end{proposition}

\begin{proof}
This is the standard Taylor expansion argument for a generic saddle-node bifurcation \cite{Kuznetsov2004}. The vanishing
of the constant and linear terms in $a$ follows from the equilibrium and fold conditions, while
nonzero $\partial_{AA}F$ and nonzero unfolding derivatives ensure a quadratic fold. A smooth change of
coordinates and rescaling brings the vector field to the stated form.
\end{proof}

Ignoring higher-order terms, the leading reduced dynamics are
\begin{equation}
\dot a = \mu-a^2.
\label{eq:snnf}
\end{equation}
For $\mu>0$, there are two equilibria, one stable and one unstable. At $\mu=0$ they collide. For
$\mu<0$, no equilibria remain, yet trajectories pass slowly through the remnant of the vanished
high-activity state.

Writing $\mu=-\delta$ with $\delta>0$ small gives
\begin{equation}
\dot a = -(a^2+\delta).
\label{eq:ghost}
\end{equation}

\begin{proposition}[Ghost residence time] \label{prop:ghost-residence-time}
Fix $M>0$. Let $a(t)$ solve \eqref{eq:ghost} with $a(0)=M$, and let $T_\delta$ be the first time such
that $a(T_\delta)=-M$. Then
\[
T_\delta=\int_{-M}^{M}\frac{da}{a^2+\delta}=\frac{2}{\sqrt{\delta}}\arctan\!\left(\frac{M}{\sqrt{\delta}}\right),
\]
and hence
\[
T_\delta\sim \frac{\pi}{\sqrt{\delta}}\qquad\text{as }\delta\downarrow 0.
\]
\end{proposition}

\begin{proof}
Equation \eqref{eq:ghost} is separable. Integrating $dt=-da/(a^2+\delta)$ from $a=M$ to $a=-M$
yields the stated formula, and the asymptotic follows from $\arctan(M/\sqrt{\delta})\to \pi/2$.
\end{proof}

The inverse-square-root scaling is characteristic of the bottleneck
associated with a generic saddle-node bifurcation \cite{Kuehn2009}.
The key point is that increasing treatment can move the activity dynamics
past the saddle-node threshold while the activity state remains for a long
but finite time near the ghost of the vanished high-activity equilibrium.
This gives a concrete dynamical mechanism for delayed pathway shutdown and
transient resistance windows \cite{Stanoev2020}.

\section{Bifurcation structure of the activity subsystem}
\label{sec:bifurcation}
Because \(B\) evolves on the slow timescale \(O(\varepsilon^{-1})\), we first
study the activity layer with \(B\) treated as a frozen parameter. To express
this layer directly in terms of inhibitor concentration, define
\begin{equation}
\widehat F(A,B,C)
:=
F\bigl(A,B;I(C)\bigr)
=
\lambda\Bigl(\widetilde{\Phi}\bigl(I(C),A,B\bigr)-A\Bigr).
\label{eq:FABC}
\end{equation}
Thus the frozen activity equation is
\(\dot A=\widehat F(A,B,C)\). By \eqref{eq:Imonotone} and
\eqref{eq:FImonotone}, \(\partial_C\widehat F=\partial_I F\,I'(C)<0\).
Consequently, concentration is a nondegenerate unfolding parameter wherever
the remaining saddle-node conditions hold.

We now characterize the qualitative structure of the reduced activity dynamics
\eqref{eq:FABI} as parameters vary.

A saddle-node bifurcation of the frozen activity equation occurs at
\((A_c,B_c,C_c)\) when
\begin{equation}
\widehat F(A_c,B_c,C_c)=0,
\qquad
\partial_A\widehat F(A_c,B_c,C_c)=0,
\qquad
\partial_{AA}\widehat F(A_c,B_c,C_c)\neq0,
\label{eq:foldcondsC}
\end{equation}
and the parameter dependence is nondegenerate:
\begin{equation}
\bigl(
\partial_B\widehat F,
\partial_C\widehat F
\bigr)(A_c,B_c,C_c)\neq(0,0).
\label{eq:foldtransversality}
\end{equation}

\begin{proposition}[Existence of a fold curve] \label{prop:fold-curve}
Let \((A_c,B_c,C_c)\) satisfy
\begin{equation}
\widehat F(A_c,B_c,C_c)=0,
\qquad
\partial_A\widehat F(A_c,B_c,C_c)=0.
\label{eq:foldcondsC2}
\end{equation}
Assume that \(\widehat F\in C^2\) near \((A_c,B_c,C_c)\) and that
\begin{equation}
\partial_{AA}\widehat F(A_c,B_c,C_c)\neq0,
\qquad
\bigl(
\partial_B\widehat F,
\partial_C\widehat F
\bigr)(A_c,B_c,C_c)\neq(0,0).
\label{eq:foldtransversality2}
\end{equation}
Then the set
\[
\widetilde{\mathcal F}
=
\left\{
(A,B,C):
\widehat F(A,B,C)=0,\ 
\partial_A\widehat F(A,B,C)=0
\right\}
\]
is locally a one-dimensional \(C^1\) manifold. Its projection onto the
\((B,C)\)-plane is a \(C^1\) fold curve \(\mathcal F\). In particular, because $\partial_C\widehat F
=
\partial_I F\,I'(C)<0$, the
transversality condition holds under
\eqref{eq:Imonotone}--\eqref{eq:FImonotone}, and the fold can locally be
written as \(A=A_{\mathrm f}(B)\), \(C=C_{\mathrm f}(B)\).
\end{proposition}

\begin{proof}
Define
\[
\mathcal G(A,B,C)
=
\begin{pmatrix}
\widehat F(A,B,C)\\
\partial_A\widehat F(A,B,C)
\end{pmatrix}.
\]
At a fold point,
\[
D\mathcal G
=
\begin{pmatrix}
0 &
\partial_B\widehat F &
\partial_C\widehat F\\
\partial_{AA}\widehat F &
\partial_{AB}\widehat F &
\partial_{AC}\widehat F
\end{pmatrix}.
\]
Conditions \eqref{eq:foldtransversality2} imply that this matrix has rank
two. The regular-level-set theorem therefore shows that
\(\mathcal G^{-1}(0)\) is locally a one-dimensional \(C^1\) manifold.

If \(\partial_C\widehat F\neq0\), the Jacobian of \(\mathcal G\) with
respect to \((A,C)\) has determinant
\(-\partial_C\widehat F\,\partial_{AA}\widehat F\neq0\).
The implicit function theorem consequently gives
\(A=A_{\mathrm f}(B)\) and \(C=C_{\mathrm f}(B)\). If instead
\(\partial_B\widehat F\neq0\), the analogous argument solves for
\((A,B)\) as functions of \(C\).
\end{proof}

\begin{proposition}[Local phase portrait near the fold]
Under the hypotheses of Proposition~\ref{prop:fold-curve}, there are neighbourhoods
\(U\) of \(A_c\) and \(\mathcal V\) of \((B_c,C_c)\), together with a
smooth signed unfolding parameter \(\mu=\mu(B,C)\), such that
\(\mathcal F\cap\mathcal V=\{(B,C)\in\mathcal V:\mu(B,C)=0\}\).
After choosing the sign of \(\mu\), the equilibria of
\(\dot A=\widehat F(A,B,C)\)
in \(U\) have the following structure:
\begin{enumerate}[label=(\roman*)]
\item if \(\mu(B,C)>0\), there are exactly two simple equilibria in
      \(U\), one unstable and one asymptotically stable;
\item if \(\mu(B,C)=0\), these equilibria coalesce in one nonhyperbolic
      saddle-node equilibrium;
\item if \(\mu(B,C)<0\), there is no equilibrium in \(U\).
\end{enumerate}
Crossing the fold therefore changes the number of equilibria in \(U\) by
two. No conclusion about equilibria outside \(U\) follows from the local
saddle-node hypotheses alone.
\end{proposition}

\begin{proof}
Proposition~\ref{prop:fold-curve} supplies a smooth signed
defining function \(\mu(B,C)\) for the fold curve, while
Proposition~\ref{prop:local-saddle-node} gives the stated local
normal form and equilibrium structure.
\end{proof}

\begin{remark} \rm
The ghost dynamics described in \eqref{eq:ghost} arise when trajectories pass
near the fold curve in the regime where equilibria have annihilated. The distance
to the fold, measured by the unfolding parameter $\mu(B,C)$, determines the
timescale of transient memory.
\end{remark}

\begin{remark} \rm
This bifurcation structure provides a geometric interpretation of dose--response curves:
increasing inhibitor concentration $C$ changes the quasi-steady occupancy input
$X_{\mathrm{qs}}(C)$ and may move the reduced activity system across a fold curve. The
resulting transition from sustained signalling to suppressed signalling can exhibit long-lived
transients when trajectories pass near the ghost of a vanished high-activity equilibrium.
\end{remark}

\subsection{Explicit fold conditions for the sigmoidal prototype}
\label{sec:explicit-sigmoid-fold}

The general fold conditions can be evaluated explicitly for
\eqref{eq:Phi}.  With the net occupancy drive
\eqref{eq:net-occupancy-drive}, the frozen activity vector field is
\begin{equation}
\widehat F_{\mathrm{sig}}(A,B,C)
=
\lambda\left[
\sigma\bigl(h(C)+\gamma A-\rho B\bigr)-A
\right].
\label{eq:sigmoid-frozen-field}
\end{equation}

\begin{proposition}[Explicit high-activity fold]
\label{prop:explicit-sigmoid-fold}
Assume that $\lambda,\rho>0$, $\gamma>4$, and
$h\in C^3$ with $h'(C)<0$. Set
\begin{equation}
d=\sqrt{1-\frac4\gamma},
\qquad
A_{\mathrm f}=\frac{1+d}{2},
\qquad
q_{\mathrm f}
=
\log\left(\frac{A_{\mathrm f}}{1-A_{\mathrm f}}\right)
-\gamma A_{\mathrm f}.
\label{eq:explicit-fold-quantities}
\end{equation}
For every $C$ such that
\begin{equation}
B_{\mathrm f}(C)
=
\frac{h(C)-q_{\mathrm f}}{\rho}
\in(0,1),
\label{eq:explicit-fold-curve}
\end{equation}
the point
$(A_{\mathrm f},B_{\mathrm f}(C),C)$ is a nondegenerate fold of
the frozen activity equation. More precisely,
\begin{align}
\partial_{AA}\widehat F_{\mathrm{sig}}
&=-\lambda\gamma d\neq0,
&
\partial_B\widehat F_{\mathrm{sig}}
&=-\frac{\lambda\rho}{\gamma}\neq0,
&
\partial_C\widehat F_{\mathrm{sig}}
&=\frac{\lambda h'(C)}{\gamma}\neq0
\label{eq:explicit-fold-derivatives}
\end{align}
at the fold. In addition,
\begin{equation}
B_{\mathrm f}'(C)=\frac{h'(C)}{\rho}<0.
\label{eq:explicit-fold-slope}
\end{equation}
The two equilibria near this fold exist on the side
$h(C)-\rho B-q_{\mathrm f}>0$; the equilibrium with
$A>A_{\mathrm f}$ is asymptotically stable in the frozen activity
equation.
\end{proposition}

\begin{proof}
At a frozen equilibrium, write
$z=h(C)+\gamma A-\rho B$. Since $A=\sigma(z)$, one has
$\sigma'(z)=A(1-A)$ and therefore
\(\partial_A\widehat F_{\mathrm{sig}}=
\lambda\bigl[\gamma A(1-A)-1\bigr]\).
The fold equation is consequently
$\gamma A(1-A)=1$. It has two distinct solutions precisely when
$\gamma>4$, and the high-activity solution is
$A_{\mathrm f}=(1+d)/2$. The equilibrium equation then gives
\eqref{eq:explicit-fold-curve}.

Using
$A_{\mathrm f}(1-A_{\mathrm f})=1/\gamma$ and
$1-2A_{\mathrm f}=-d$, direct differentiation gives
\eqref{eq:explicit-fold-derivatives}. Thus both the quadratic fold
condition and concentration transversality hold. Formula
\eqref{eq:explicit-fold-slope} follows by differentiating
\eqref{eq:explicit-fold-curve}.

Finally, set
\(\Delta=h(C)-\rho B-q_{\mathrm f}\). Taylor expansion of the
equilibrium equation about the fold shows that two nearby roots exist
for $\Delta>0$. On the upper root,
$\partial_A\widehat F_{\mathrm{sig}}<0$, which proves frozen
asymptotic stability.
\end{proof}

\begin{remark}[A nonempty parameter regime]
\label{rem:explicit-parameter-example}
For example, take
\(\gamma=\rho=6\), \(\lambda=\kappa_A=1\), \(\kappa_C=2\), and
\(h(C)=1.184907-2C/(1+C)\).
At $C=1$, formulas \eqref{eq:explicit-fold-quantities}--
\eqref{eq:explicit-fold-curve} give
\(A_{\mathrm f}\approx0.788675\) and
\(B_{\mathrm f}\approx0.600000\). The transverse adaptive
condition derived in Proposition~\ref{prop:explicit-transverse-fold}
below is also satisfied for this parameter set.
\end{remark}

\section{From the frozen fold to the coupled fast--slow dynamics}
\label{sec:coupled-fold}

The preceding analysis treats \(B\) as a frozen parameter and characterizes
the fold geometry of the activity equation $\dot A=\widehat F(A,B,C)$.
This approximation is appropriate on the fast activity timescale because
\(B\) evolves on the slower timescale \(O(\varepsilon^{-1})\). In the
singular limit \(\varepsilon=0\), the set
\(\mathcal C_0=\bigl\{(A,B,C):\widehat F(A,B,C)=0\bigr\}\)
is the critical manifold of the activity--adaptation system, and the fold
curve \(\mathcal F\) identified in the preceding section is the subset on
which \(\partial_A\widehat F(A,B,C)=0\).

For \(0<\varepsilon\ll1\), however, \(B\) is not constant. The coupled
reduced system is
\begin{align}
\dot A
&=
\widehat F(A,B,C),\\
\dot B
&=
\varepsilon\bigl(\Psi(A,C)-B\bigr).
\end{align}
Consequently, the signed unfolding parameter \(\mu(B,C)\) varies along a
trajectory as \(B\) adapts. The frozen-\(B\) ghost estimate must therefore
be related to the slow drift of \(B\) before it can be used to describe the
coupled dynamics.

Two distinct questions arise. First, a fold of the frozen activity
subsystem is not necessarily a saddle-node equilibrium of the coupled
activity--adaptation system, because a coupled equilibrium must also
satisfy \(B=\Psi(A,C)\). Second, even when the coupled system is not at an
equilibrium saddle-node, a trajectory may pass slowly through a
neighbourhood of the frozen fold as \(B\) evolves. We examine these two
issues separately below.

\subsection{Equilibria of the coupled activity--adaptation system}

We first distinguish a fold of the frozen activity equation from a
saddle-node equilibrium of the coupled system. For fixed \(C\) and
\(\varepsilon>0\), an equilibrium \((A^*,B^*)\) of
\(\dot A=\widehat F(A,B,C)\),
\(\dot B=\varepsilon\bigl(\Psi(A,C)-B\bigr)\) must satisfy
\(\widehat F(A^*,B^*,C)=0\) and \(B^*=\Psi(A^*,C)\).
It is therefore convenient to define the scalar equilibrium residual
\begin{equation}
\mathcal E(A,C)
=
\widehat F\bigl(A,\Psi(A,C),C\bigr).
\label{eq:coupled-equilibrium-residual}
\end{equation}
The coupled equilibria are precisely the solutions of
\(\mathcal E(A,C)=0\) and \(B=\Psi(A,C)\).

Differentiating \eqref{eq:coupled-equilibrium-residual} with respect to
\(A\) gives
\begin{equation}
\partial_A\mathcal E
=
\partial_A\widehat F
+
\partial_B\widehat F\,\partial_A\Psi,
\label{eq:coupled-fold-derivative}
\end{equation}
where all derivatives on the right-hand side are evaluated at
\((A,\Psi(A,C),C)\). Similarly,
\(\partial_C\mathcal E=\partial_C\widehat F+
\partial_B\widehat F\,\partial_C\Psi\),
and
\[
\partial_{AA}\mathcal E
=
\partial_{AA}\widehat F
+
2\partial_{AB}\widehat F\,\partial_A\Psi
+
\partial_{BB}\widehat F\,(\partial_A\Psi)^2
+
\partial_B\widehat F\,\partial_{AA}\Psi.
\]
These identities show explicitly how adaptive feedback modifies the
location, transversality, and curvature of a fold of the equilibrium
branch.

\begin{proposition}[Saddle-node criterion for the coupled system]
Assume that \(\widehat F\) and \(\Psi\) are \(C^3\) near
\((A_c,B_c,C_c)\). Suppose that \(B_c=\Psi(A_c,C_c)\) and that
\begin{equation}
\mathcal E(A_c,C_c)=0,
\qquad
\partial_A\mathcal E(A_c,C_c)=0.
\label{eq:coupled-fold-conditions}
\end{equation}
Assume further that \(\partial_{AA}\mathcal E(A_c,C_c)\neq0\) and
\(\partial_C\mathcal E(A_c,C_c)\neq0\), and that
\begin{equation}
\operatorname{tr}J_c
=
\partial_A\widehat F(A_c,B_c,C_c)-\varepsilon
<0,
\label{eq:coupled-trace}
\end{equation}
where
\[
J_c
=
\begin{pmatrix}
\partial_A\widehat F(A_c,B_c,C_c)
&
\partial_B\widehat F(A_c,B_c,C_c)
\\
\varepsilon\partial_A\Psi(A_c,C_c)
&
-\varepsilon
\end{pmatrix}.
\]
Then \((A_c,B_c)\) is a nondegenerate saddle-node equilibrium of the
coupled activity--adaptation system at \(C=C_c\).

More precisely, there is a one-dimensional centre manifold through
\((A_c,B_c)\), together with a smooth local coordinate \(u\) and a smooth
signed parameter \(\nu=\nu(C)\), satisfying
\(\nu(C_c)=0\) and \(\nu'(C_c)\neq0\),
such that the dynamics on the centre manifold are locally equivalent to
\begin{equation}
\dot u
=
\nu-u^2
+
O\bigl(|u|^3+|\nu||u|+\nu^2\bigr).
\label{eq:coupled-normal-form}
\end{equation}
After choosing the sign of \(\nu\), the coupled system has two nearby
equilibria for \(\nu>0\), one asymptotically stable and one a saddle; one
nonhyperbolic saddle-node equilibrium for \(\nu=0\); and no nearby
equilibrium for \(\nu<0\).
\end{proposition}

\begin{proof}
The equilibrium equation for \(B\) gives \(B=\Psi(A,C)\). Substitution
into the activity equation reduces the coupled equilibrium problem to
the scalar equation \(\mathcal E(A,C)=0\).
The Jacobian of the coupled system at \((A_c,B_c,C_c)\) is \(J_c\), and
its determinant satisfies
\begin{align*}
\det J_c
&=
-\varepsilon
\left[
\partial_A\widehat F(A_c,B_c,C_c)
+
\partial_B\widehat F(A_c,B_c,C_c)
\partial_A\Psi(A_c,C_c)
\right]\\
&=
-\varepsilon\partial_A\mathcal E(A_c,C_c)
=
0.
\end{align*}
Thus \(J_c\) has a zero eigenvalue. Condition
\eqref{eq:coupled-trace} shows that the second eigenvalue is nonzero and
negative, so the zero eigenvalue is simple and the transverse direction
is asymptotically stable.

The conditions \(\partial_{AA}\mathcal E(A_c,C_c)\neq0\) and
\(\partial_C\mathcal E(A_c,C_c)\neq0\) are the quadratic
nondegeneracy and parameter-transversality conditions
for a generic saddle-node. The centre-manifold theorem and the standard
saddle-node normal-form reduction therefore give
\eqref{eq:coupled-normal-form} after smooth changes of state, parameter,
and time \cite{Kuznetsov2004}. The local equilibrium structure follows from
the normal form and the negativity of the transverse eigenvalue.
\end{proof}

The coupled fold condition differs from the frozen-\(B\) condition used
in the preceding section. The frozen activity subsystem has a fold when
\[
\partial_A\widehat F(A_c,B_c,C_c)=0,
\]
whereas the coupled equilibrium branch has a fold when
\[
\partial_A\widehat F(A_c,B_c,C_c)
+
\partial_B\widehat F(A_c,B_c,C_c)
\partial_A\Psi(A_c,C_c)
=
0.
\]
The two conditions coincide only when 
\[
\partial_B\widehat F(A_c,B_c,C_c)
\partial_A\Psi(A_c,C_c)=0.
\]
Thus the fold curve \(\mathcal F\) of the frozen activity subsystem should
not in general be identified with the saddle-node set of the coupled
activity--adaptation system. Nevertheless, the frozen fold remains
dynamically relevant because the slowly evolving variable \(B\) may
carry trajectories through its neighbourhood. This slow-passage problem
is considered next. Figure~\ref{fig:frozen-fold-geometry} illustrates these two
geometries and the possibility of slow passage through a frozen fold.

\begin{figure}[htbp]
\centering
\includegraphics[width=0.97\textwidth]{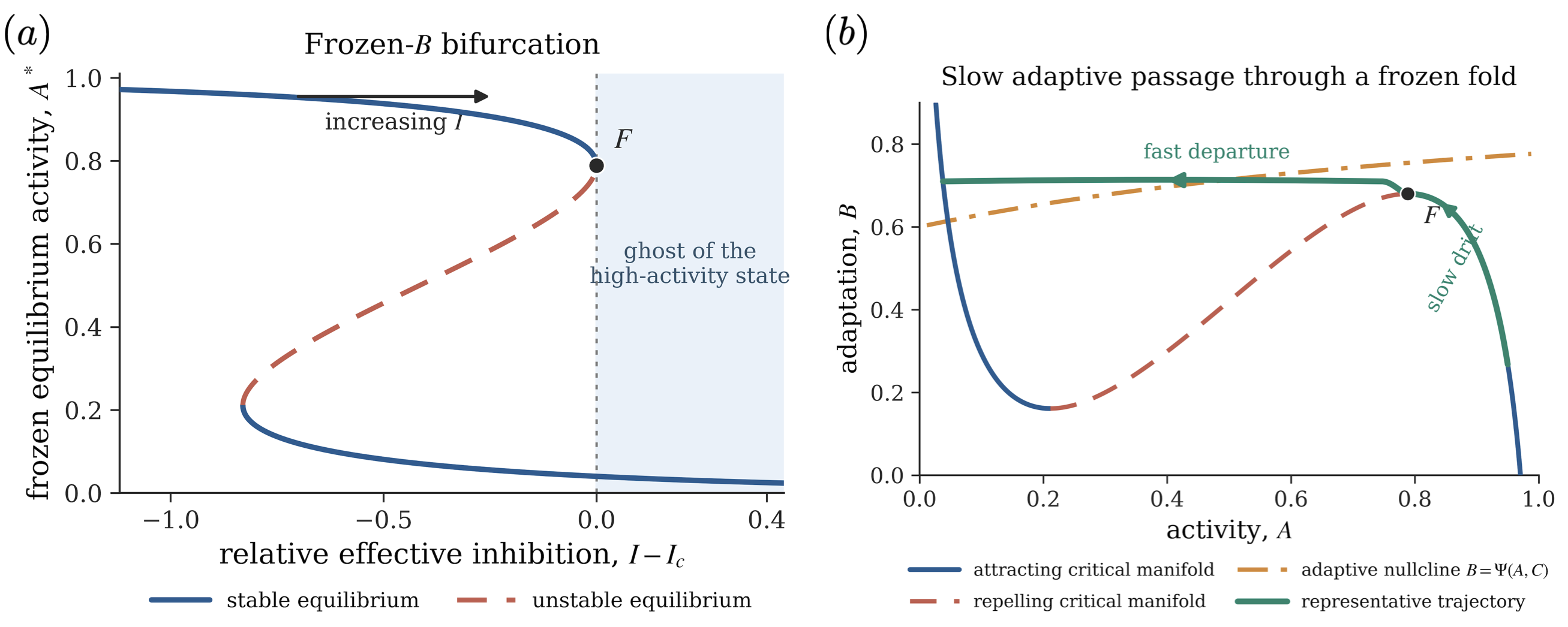}
\caption{Representative fold geometry of the reduced
activity--adaptation model. The parameter values are illustrative
and are not fitted to the viability data.
\textup{(a)} Frozen-\(B\) equilibria of the activity equation as
effective inhibition \(I\) increases. Solid and dashed curves denote
stable and unstable equilibria, respectively. The high-activity and
unstable branches meet at the frozen fold \(F\); beyond it, the
trajectory enters the ghost region.
\textup{(b)} Critical manifold and adaptive nullcline at fixed \(C\).
The representative trajectory drifts toward and passes through the
frozen fold before undergoing fast departure. Because the fold does
not generally lie on \(B=\Psi(A,C)\), it need not be a saddle-node
equilibrium of the coupled system.}
\label{fig:frozen-fold-geometry}
\end{figure}

\subsection{Slow adaptive passage through the frozen fold}

We now consider the effect of the slow evolution of \(B\) on trajectories
passing near the fold of the frozen activity subsystem. Let
\(\mu=\mu(B,C)\) be the signed unfolding parameter introduced in the
preceding section, with \(\mu>0\) on the side with two frozen equilibria
and \(\mu<0\) on the side without equilibria. For fixed \(C\), differentiation
along a trajectory of the coupled system gives
\begin{equation}
\dot\mu
=
\partial_B\mu(B,C)\dot B
=
\varepsilon\partial_B\mu(B,C)
\bigl(\Psi(A,C)-B\bigr).
\label{eq:mu-drift}
\end{equation}
Thus \(\mu\) is constant only in the singular limit
\(\varepsilon=0\), or when the adaptive vector field vanishes.

In the local coordinates and normalized time used in the saddle-node
reduction, the coupled dynamics near a frozen fold can be written as
\begin{align}
\dot a
&=
\mu-a^2+\mathcal R(a,\mu,\varepsilon),
\label{eq:slow-fold-a}\\
\dot\mu
&=
\varepsilon\mathcal Q(a,\mu,C),
\label{eq:slow-fold-mu}
\end{align}
where
\(\mathcal R(a,\mu,\varepsilon)=
O\bigl(|a|^3+|\mu||a|+\mu^2+\varepsilon\bigr)\) and
\(\mathcal Q(a,\mu,C)=\partial_B\mu(B,C)
\bigl(\Psi(A,C)-B\bigr)\),
with \(A\) and \(B\) expressed in the local coordinates. The
\(O(\varepsilon)\) term in \(\mathcal R\) accounts for the fact that the
coordinate transformation used in the frozen-parameter normal form may
itself depend on the slowly varying parameter \(\mu\).

We first identify the regime in which the frozen-\(B\) ghost estimate
remains valid.

\begin{proposition}[Validity of the frozen ghost law under slow adaptation]
\label{prop:frozen-ghost-law}
Fix \(M>0\) sufficiently small, and suppose that \(\mathcal Q\) remains
bounded in the local fold neighbourhood. Let
\(T_{\delta,\varepsilon}\) be the first time at which a trajectory of
\eqref{eq:slow-fold-a}--\eqref{eq:slow-fold-mu}, starting from
\(a(0)=M\), \(\mu(0)=-\delta\), \(\delta>0\),
reaches \(a=-M\). If
\begin{equation}
\frac{\varepsilon}{\delta^{3/2}}
\longrightarrow0
\qquad
\text{as }\delta\downarrow0,
\label{eq:frozen-ghost-condition}
\end{equation}
then
\begin{equation}
T_{\delta,\varepsilon}
\sim
\frac{\pi}{\sqrt{\delta}}.
\label{eq:slow-frozen-ghost}
\end{equation}
\end{proposition}

\begin{proof}
Let \(Q_0>0\) be a uniform bound for \(\mathcal Q\) in the local fold
neighbourhood. As long as the trajectory remains in this neighbourhood,
\eqref{eq:slow-fold-mu} gives $|\mu(t)+\delta|
\le
\varepsilon Q_0t$. We first establish the passage-time bound
$T_{\delta,\varepsilon}=O(\delta^{-1/2})$. Because \(M>0\) is sufficiently
small and \(\mathcal R(a,\mu,\varepsilon)=
O\bigl(|a|^3+|\mu||a|+\mu^2+\varepsilon\bigr)\),
a comparison argument, initially under the bootstrap assumption
\(|\mu(t)+\delta|\le\delta/2\), gives constants \(c_1,c_2>0\),
independent of \(\delta\), such that
\[
-c_2(a^2+\delta)
\le
\dot a
\le
-c_1(a^2+\delta)
\]
throughout the passage. Integration therefore gives $T_{\delta,\varepsilon}
=
O(\delta^{-1/2})$. Consequently,
\[
\sup_{0\le t\le T_{\delta,\varepsilon}}
|\mu(t)+\delta|
=
O(\varepsilon\delta^{-1/2})
=
o(\delta)
\]
by \eqref{eq:frozen-ghost-condition}. This closes the bootstrap and shows
that $\mu(t)=-\delta+o(\delta)$ uniformly during the bottleneck passage.

Introduce the scaled variables \(a=\sqrt{\delta}\,y\) and
\(s=\sqrt{\delta}\,t\). On every fixed bounded \(y\)-interval, the
activity equation becomes \(dy/ds=-(1+y^2)+o(1)\),
uniformly as \(\delta\downarrow0\). Hence, for every fixed \(L>0\), the
scaled time required to pass from \(y=L\) to \(y=-L\) converges to
\(\int_{-L}^{L}(1+y^2)^{-1}\,dy=2\arctan L\).

Standard comparison estimates in the regions
\(L\sqrt{\delta}\le |a|\le M\) show that their combined contribution to
the scaled passage time is \(O(L^{-1})\), uniformly for sufficiently
small \(\delta\). Taking first \(\delta\downarrow0\) and then
\(L\to\infty\), we obtain
\[
\sqrt{\delta}\,T_{\delta,\varepsilon}
\longrightarrow
\int_{-\infty}^{\infty}\frac{dy}{1+y^2}
=
\pi.
\]
Therefore \(T_{\delta,\varepsilon}\sim\pi/\sqrt{\delta}\).
\end{proof}

When $\varepsilon/\delta^{3/2}$ is not asymptotically small,
Proposition~\ref{prop:frozen-ghost-law} no longer guarantees that
the frozen-$B$ approximation is uniform through the bottleneck.
The resulting behaviour depends on the direction and magnitude
of the adaptive drift. We next consider the case of a transverse
adaptive crossing of the frozen fold.

After choosing the orientation of $\mu$ and absorbing the
positive factor induced by the normal-form time rescaling, define
\[
\gamma_c
=
-\vartheta_c\,
\partial_B\mu(B_c,C_c)
\bigl(\Psi(A_c,C_c)-B_c\bigr),
\qquad
\vartheta_c>0,
\]
where $\vartheta_c$ is the time-normalization factor. A
transverse adaptive crossing corresponds to $\gamma_c>0$.

For the sigmoidal prototype, this condition and the normal-form
coefficient can be written entirely in terms of the model parameters.

\begin{proposition}[Explicit transverse passage for the sigmoidal model]
\label{prop:explicit-transverse-fold}
Fix $C=C_{\mathrm f}$, and suppose that the hypotheses of
Proposition~\ref{prop:explicit-sigmoid-fold} hold. Define
\begin{equation}
D_{\mathrm f}
=
\Psi(A_{\mathrm f},C_{\mathrm f})
-B_{\mathrm f}(C_{\mathrm f}).
\label{eq:explicit-transverse-condition}
\end{equation}
If
\begin{equation}
D_{\mathrm f}>0,
\label{eq:explicit-transverse-positive}
\end{equation}
or, equivalently,
\begin{equation}
\frac{\kappa_AA_{\mathrm f}+\kappa_CC_{\mathrm f}}
{1+\kappa_AA_{\mathrm f}+\kappa_CC_{\mathrm f}}
>
\frac{h(C_{\mathrm f})-q_{\mathrm f}}{\rho},
\label{eq:explicit-transverse-parameters}
\end{equation}
then the slow adaptive flow crosses the high-activity frozen fold
transversely from the side with two frozen equilibria to the side with
no nearby frozen equilibrium.

More precisely, let
\begin{equation}
k_{\mathrm f}=\frac{\lambda\gamma d}{2},
\qquad
\Delta=h(C_{\mathrm f})-\rho B-q_{\mathrm f},
\qquad
a=k_{\mathrm f}(A-A_{\mathrm f}),
\qquad
\mu=\frac{k_{\mathrm f}\lambda}{\gamma}\Delta.
\label{eq:explicit-normal-form-coordinates}
\end{equation}
Then the local system has the form
\begin{align}
\dot a
&=
\mu-a^2
+O\bigl(|a|^3+|a||\mu|+\mu^2\bigr),
\label{eq:explicit-normal-form-a}\\
\dot\mu
&=
-\varepsilon\gamma_c
+O\bigl(\varepsilon(|a|+|\mu|)\bigr),
\label{eq:explicit-normal-form-mu}
\end{align}
where
\begin{equation}
\gamma_c
=
\frac{k_{\mathrm f}\lambda\rho}{\gamma}D_{\mathrm f}
=
\frac{\lambda^2\rho d}{2}D_{\mathrm f}
>0.
\label{eq:explicit-gamma-c}
\end{equation}
If a trajectory starts sufficiently close to the upper attracting
frozen branch with $B<B_{\mathrm f}(C_{\mathrm f})$, then it tracks
that branch into the fold neighbourhood and satisfies the tracking
hypothesis of Theorem~\ref{thm:adaptive-delay}.
\end{proposition}

\begin{proof}
Along the upper frozen branch one has $A_{\mathrm s}(B)>A_{\mathrm f}$
and $B<B_{\mathrm f}$. Since $\Psi$ is increasing in $A$,
\[
\Psi(A_{\mathrm s}(B),C_{\mathrm f})-B
\geq
\Psi(A_{\mathrm f},C_{\mathrm f})
-B_{\mathrm f}(C_{\mathrm f})
=D_{\mathrm f}>0.
\]
Thus $B$ increases and
\(\dot\Delta=-\varepsilon\rho(\Psi-B)<0\), so the slow flow is
directed transversely through $\Delta=0$. Away from the fold, the
upper branch is normally hyperbolic and attracting. Standard
slow-manifold tracking therefore carries trajectories that start near
this branch into the fold neighbourhood.

To obtain the coefficients, observe that
\[
h(C_{\mathrm f})+\gamma A-\rho B
=
\log\left(\frac{A_{\mathrm f}}{1-A_{\mathrm f}}\right)
+\gamma(A-A_{\mathrm f})+\Delta.
\]
Taylor expansion of the sigmoid at the fold, followed by the scaling
\eqref{eq:explicit-normal-form-coordinates}, gives
\eqref{eq:explicit-normal-form-a}. Moreover,
\(\dot\mu=-\varepsilon(k_{\mathrm f}\lambda\rho/\gamma)
\bigl(\Psi(A,C_{\mathrm f})-B\bigr)\).
Expanding the last factor about the fold gives
\eqref{eq:explicit-normal-form-mu} and
\eqref{eq:explicit-gamma-c}.
\end{proof}

For the parameters in Remark~\ref{rem:explicit-parameter-example},
\(\Psi(A_{\mathrm f},1)\approx0.73605\),
\(D_{\mathrm f}\approx0.13605\), and \(\gamma_c\approx0.23565\).
Thus the explicit fold and transversality conditions are simultaneously
satisfied in a biologically admissible state region.

Near such a crossing, the normalized local system takes the form
\begin{align}
\dot a
&=
\mu-a^2+R(a,\mu,\varepsilon),
\label{eq:dynamic-fold-a}\\
\dot\mu
&=
-\varepsilon\gamma_c
+\varepsilon\widetilde S(a,\mu,\varepsilon),
\label{eq:dynamic-fold-mu}
\end{align}
where
$\widetilde S(a,\mu,\varepsilon)
=O(|a|+|\mu|+\varepsilon)$.

\begin{theorem}[Delay under transverse adaptive passage]
\label{thm:adaptive-delay}
Consider \eqref{eq:dynamic-fold-a}--\eqref{eq:dynamic-fold-mu},
where $\gamma_c>0$ is independent of $\varepsilon$. Assume that
$R$ and $\widetilde S$ are locally $C^2$ and, uniformly in a
fixed fold neighbourhood,
\[
R=O\bigl(|a|^3+|\mu||a|+\mu^2+\varepsilon\bigr),
\qquad
\widetilde S=O\bigl(|a|+|\mu|+\varepsilon\bigr)
\]
as $(a,\mu,\varepsilon)\to(0,0,0)$. Shrink the neighbourhood,
if necessary, so that
$|\widetilde S|\leq\gamma_c/2$ for all sufficiently small
$\varepsilon$.

Suppose that a trajectory approaches the fold while tracking the
attracting equilibrium branch of the frozen activity subsystem.
Let $t_c$ be the unique time at which $\mu(t_c)=0$, set
$\omega=\varepsilon\gamma_c$, and define
$y_\varepsilon(\tau)=\omega^{-1/3}
a(t_c+\omega^{-1/3}\tau)$ and
$m_\varepsilon(\tau)=\omega^{-2/3}
\mu(t_c+\omega^{-1/3}\tau)$. Assume that
\[
\lim_{T\to\infty}\limsup_{\varepsilon\downarrow0}
\left(
|m_\varepsilon(-T)-T|
+
|y_\varepsilon(-T)-\sqrt T|
\right)=0.
\]
Then, uniformly on compact subsets of
$(-\infty,z_{\mathrm{Ai}})$,
\[
m_\varepsilon(\tau)\longrightarrow-\tau,
\qquad
y_\varepsilon(\tau)\longrightarrow
-\frac{\operatorname{Ai}'(-\tau)}
       {\operatorname{Ai}(-\tau)},
\]
where $z_{\mathrm{Ai}}\approx2.338107$ is the smallest positive
number satisfying $\operatorname{Ai}(-z_{\mathrm{Ai}})=0$.

Fix a sufficiently small $a_{\mathrm{out}}>0$ and define
$t_{\mathrm{exit}}$ as the first time after $t_c$ for which
$a(t_{\mathrm{exit}})=-a_{\mathrm{out}}$. Then
\begin{align}
t_{\mathrm{exit}}-t_c
&=
z_{\mathrm{Ai}}\omega^{-1/3}
+o(\omega^{-1/3}),
\label{eq:exit-time}\\
\mu(t_{\mathrm{exit}})
&=
-z_{\mathrm{Ai}}\omega^{2/3}
+o(\omega^{2/3}).
\label{eq:exit-mu}
\end{align}
The leading-order terms are independent of the particular fixed
choice of $a_{\mathrm{out}}$.
\end{theorem}

\begin{proof}
Under the stated rescaling, the system becomes
$y_\varepsilon'=m_\varepsilon-y_\varepsilon^2+r_\varepsilon$
and $m_\varepsilon'=-1+s_\varepsilon$, where
\[
r_\varepsilon(y,m)
=
\omega^{-2/3}
R(\omega^{1/3}y,\omega^{2/3}m,\varepsilon),
\qquad
s_\varepsilon(y,m)
=
\gamma_c^{-1}
\widetilde S(\omega^{1/3}y,\omega^{2/3}m,\varepsilon).
\]
The remainder assumptions imply that
$r_\varepsilon,s_\varepsilon\to0$ uniformly on compact sets.
Since $m_\varepsilon(0)=0$, it follows that
$m_\varepsilon(\tau)\to-\tau$. The limiting activity equation is
therefore $y'=-\tau-y^2$.

Writing $y=u'/u$ gives $u''+\tau u=0$. The tracking condition
selects $u(\tau)=\operatorname{Ai}(-\tau)$, and hence
\(y_0(\tau)=-\operatorname{Ai}'(-\tau)/\operatorname{Ai}(-\tau)\).
Continuous dependence gives $y_\varepsilon\to y_0$ uniformly on
compact subsets of $(-\infty,z_{\mathrm{Ai}})$.

To identify the exit time, fix $L>0$ and let $\tau_L$ be the
unique time satisfying $y_0(\tau_L)=-L$. Since
$y_0(\tau)\to-\infty$ as $\tau\uparrow z_{\mathrm{Ai}}$, one has
$\tau_L\uparrow z_{\mathrm{Ai}}$ as $L\to\infty$. If
$\tau_{\varepsilon,L}$ denotes the first scaled time at which
$y_\varepsilon=-L$, compact convergence gives
$\tau_{\varepsilon,L}\to\tau_L$ for every fixed $L$.

After this time, the trajectory passes from
$a=-L\omega^{1/3}$ to $a=-a_{\mathrm{out}}$. For $L$ sufficiently
large and $\varepsilon$ sufficiently small, comparison with
$\dot a=-a^2$ gives
\[
0
\leq
\omega^{1/3}
\bigl(
t_{\mathrm{exit}}
-t_c-\omega^{-1/3}\tau_{\varepsilon,L}
\bigr)
\leq
\frac{C}{L},
\]
where $C$ is independent of $L$ and $\varepsilon$. Taking first
$\varepsilon\downarrow0$ and then $L\to\infty$ yields
$\omega^{1/3}(t_{\mathrm{exit}}-t_c)\to z_{\mathrm{Ai}}$, proving
\eqref{eq:exit-time}.

Finally, integration of \eqref{eq:dynamic-fold-mu} gives
\[
\mu(t_{\mathrm{exit}})
=
-\omega(t_{\mathrm{exit}}-t_c)
+
\varepsilon
\int_{t_c}^{t_{\mathrm{exit}}}
\widetilde S(a(s),\mu(s),\varepsilon)\,ds .
\]
During the scaled bottleneck the integral contributes $O(\varepsilon)$.
During the subsequent fast exit, comparison with
$\dot a=-a^2(1+o(1))$ gives an additional contribution
$O(\varepsilon|\log\omega|)$. Since $\gamma_c$ is fixed,
$\varepsilon|\log\omega|=o(\omega^{2/3})$. Consequently,
\(\varepsilon\int_{t_c}^{t_{\mathrm{exit}}}\widetilde S\,ds
=o(\omega^{2/3})\),
and \eqref{eq:exit-mu} follows from \eqref{eq:exit-time}.
\end{proof}

\subsection{Transfer to the occupancy system: a directional estimate}
\label{sec:directional-transfer}

The Euclidean one-sided-growth criterion of
Theorem~\ref{thm:quasi-steady-validity}\textup{(ii)} is useful as a
general sufficient condition, but it is too strong for a transverse
adaptive fold. This can be seen directly for the sigmoidal prototype.
At the fold, the Jacobian of the reduced $(A,B)$ system is
\begin{equation}
J_{\mathrm f}
=
\begin{pmatrix}
0 & -\lambda\rho/\gamma\\
\varepsilon\Psi_A(A_{\mathrm f},C_{\mathrm f}) & -\varepsilon
\end{pmatrix}.
\label{eq:fold-jacobian-sigmoid}
\end{equation}
The least Euclidean one-sided Lipschitz rate is the largest eigenvalue
of the symmetric part of $J_{\mathrm f}$, namely
\begin{equation}
\ell_{\mathrm f}
=
\frac{-\varepsilon+
\sqrt{\varepsilon^2+
\left(-\lambda\rho/\gamma
+\varepsilon\Psi_A(A_{\mathrm f},C_{\mathrm f})\right)^2}}{2}
\longrightarrow
\frac{\lambda\rho}{2\gamma}>0.
\label{eq:euclidean-rate-obstruction}
\end{equation}
Thus an all-direction Euclidean logarithmic-norm estimate accumulates
over the diverging fold-passage interval and cannot provide a uniform
bound. The reduction can nevertheless be transferred because occupancy
error enters only the activity equation and is exponentially localized
near the initial time. The relevant estimate is therefore directional.

To state the result, write the full sigmoidal activity map as
\begin{equation}
\Phi(X,A,B)
=
\sigma\bigl(H(X;C)+\gamma A-\rho B\bigr),
\qquad
H(X_{\mathrm{qs}}(C);C)=h(C),
\label{eq:full-sigmoid-map}
\end{equation}
and assume that $H$ is locally Lipschitz in $X$, uniformly on the
parameter interval considered. The occupancy perturbation of the
downstream vector field then has the form
\begin{equation}
p_\eta(t)=p_{\eta,A}(t)(1,0)^T,
\qquad
|p_{\eta,A}(t)|\leq P e^{-\kappa t/\eta}
\label{eq:directional-occupancy-forcing}
\end{equation}
for a constant $P>0$.

\begin{proposition}[Model-specific transfer of the fold-delay laws]
\label{prop:model-specific-transfer}
Assume fast occupancy relaxation, the sigmoidal model
\eqref{eq:full-sigmoid-map}, and the adaptive law
\eqref{eq:Psi} with \(\kappa_A>0\). Assume also the explicit fold and transverse-crossing
conditions of Propositions~\ref{prop:explicit-sigmoid-fold} and
\ref{prop:explicit-transverse-fold}. Full and reduced solutions are
started with the same downstream state.

In the frozen-ghost regime, if
\[
\frac{\varepsilon}{\delta^{3/2}}\longrightarrow0,
\qquad
\frac{\eta}{\delta}\longrightarrow0,
\]
then the full-system passage time satisfies
\begin{equation}
T_{\delta,\varepsilon}^{\eta}
\sim\frac{\pi}{\sqrt\delta}.
\label{eq:full-frozen-directional}
\end{equation}

In the transverse dynamic-fold regime, suppose that the common initial
point lies on a fixed incoming section of the upper attracting branch,
at positive distance from the fold. Set
\(\omega=\varepsilon\gamma_c\), assume
\begin{equation}
\omega\longrightarrow0,
\qquad
\frac{\eta}{\omega^{2/3}}\longrightarrow0,
\label{eq:directional-scale-condition}
\end{equation}
and let $(t_c^\eta,t_{\mathrm{exit}}^\eta)$ be the corresponding
full-system crossing and exit times. Then
\begin{align}
t_{\mathrm{exit}}^\eta-t_c^\eta
&=
z_{\mathrm{Ai}}\omega^{-1/3}
+o\bigl(\omega^{-1/3}\bigr),
\label{eq:full-dynamic-directional-time}\\
\mu^\eta(t_{\mathrm{exit}}^\eta)
&=
-z_{\mathrm{Ai}}\omega^{2/3}
+o\bigl(\omega^{2/3}\bigr).
\label{eq:full-dynamic-directional-mu}
\end{align}
\end{proposition}

\begin{proof}
Let
\(e_A=A^\eta-\bar A\) and
\(e_B=B^\eta-\bar B\).

In the frozen regime, the smooth fold-coordinate transformation maps
the occupancy term to a scalar forcing \(r_\eta(t)\) satisfying
\(\int_0^\infty |r_\eta(t)|\,dt\leq C\eta\). Choose
\(t_b=2\eta|\log\delta|/\kappa\). During
\([0,t_b]\), variation of constants on the fixed entry section gives an
\(O(\eta)\) displacement in \(a\), an \(O(\varepsilon\eta)\)
displacement in the slow coordinate, and
\(\sqrt\delta\,t_b=o(1)\). For \(t\geq t_b\),
\(|r_\eta(t)|\leq C\delta^2\). Thus, after the scaling
\(a=\sqrt\delta\,y\), \(s=\sqrt\delta\,t\), the boundary layer changes
the initial value of \(y\) by
\(O(\eta/\sqrt\delta)=o(1)\), whereas the remaining forcing is uniformly
\(o(1)\) in the scaled equation. The slow drift over the passage is
\(O(\varepsilon\delta^{-1/2})=o(\delta)\), exactly as in
Proposition~\ref{prop:frozen-ghost-law}. Its scaled comparison argument
therefore applies with perturbed initial data and proves
\eqref{eq:full-frozen-directional}.

We next establish the estimates needed in the dynamic regime. Put
\(\Delta(t)=h(C_{\mathrm f})-\rho\bar B(t)-q_{\mathrm f}\), fix a
large \(L>1\), and let \(t_L\) be the first reduced time for which
\(\Delta(t_L)=L\omega^{2/3}\). The strict inequality
\(D_{\mathrm f}>0\) and normal hyperbolicity of the upper branch imply
that, after shrinking a fixed fold neighbourhood, there are positive
constants \(d_0,d_1,c_0,c_1\), independent of
\(\varepsilon,\eta\), such that
\[
d_0\varepsilon\leq-\dot\Delta(t)\leq d_1\varepsilon,
\qquad
c_0\sqrt{\Delta(t)}
\leq-\partial_A\widehat F_{\mathrm{sig}}(\bar A,\bar B,C_{\mathrm f})
\leq c_1\sqrt{\Delta(t)}
\]
before \(t_L\). These constants may be chosen uniformly for initial
points in a compact subset of the fixed incoming section.

Apply the mean-value theorem on the segment joining the full and reduced
downstream states. On any sufficiently small bootstrap tube about the
upper branch, the error equations have the exact form
\[
\dot e_A=-m_\eta(t)e_A-\beta_\eta(t)e_B+p_{\eta,A}(t),
\qquad
\dot e_B=\varepsilon\bigl(\zeta_\eta(t)e_A-e_B\bigr),
\]
where \(m_\eta\asymp\sqrt\Delta\), and
\(0<\beta_0\leq\beta_\eta\leq\beta_1\) and
\(0<\zeta_0\leq\zeta_\eta\leq\zeta_1\). Here positivity follows from
\(-\partial_B\widehat F_{\mathrm{sig}}>0\) and
\(\Psi_A>0\). On the outer portion
\(\Delta\geq\Delta_0>0\), this system is uniformly normally
hyperbolic. Duhamel's formula, using
\eqref{eq:directional-occupancy-forcing}, therefore gives at entry into
the fixed fold neighbourhood
\[
|e_A|\leq C\varepsilon\eta,
\qquad
|e_B|\leq C\varepsilon\eta.
\]
The factor \(\varepsilon\) in the second estimate is the projection of
the activity impulse through the equation for \(B\); the fast component
of \(e_A\) has decayed before this entry.

For completeness, the estimate remains uniform as normal hyperbolicity
weakens. Let
\(\beta_{\mathrm f}=-\partial_B\widehat F_{\mathrm{sig}}
(A_{\mathrm f},B_{\mathrm f},C_{\mathrm f})\),
\(\zeta_{\mathrm f}=\Psi_A(A_{\mathrm f},C_{\mathrm f})\), and
\(q_0=\beta_{\mathrm f}/\zeta_{\mathrm f}>0\). Smoothness and the fold
expansion give
\(|q_0\zeta_\eta-\beta_\eta|\leq C m_\eta\). For the directional
energy
\(\mathcal E_{\mathrm d}=e_A^2+(q_0/\varepsilon)e_B^2\), direct differentiation
along the error system yields
\[
\dot{\mathcal E}_{\mathrm d}
=-2m_\eta e_A^2-2q_0e_B^2
+2(q_0\zeta_\eta-\beta_\eta)e_Ae_B
+2e_Ap_{\eta,A}.
\]
Shrinking the fold neighbourhood absorbs the mixed term. The occupancy
forcing is already exponentially small on this part of the trajectory,
so integration gives
\(\mathcal E_{\mathrm d}(t)\leq C\varepsilon\eta^2\) for \(t\leq t_L\).
Consequently,
\begin{align}
|e_B(t)|
&\leq C\varepsilon\eta,
\label{eq:directional-error-B}\\
|e_A(t)|
&\leq
C\eta
\exp\left(-c\int_0^t\sqrt{\Delta(s)}\,ds\right)
+C\frac{\varepsilon\eta}
{\sqrt{\Delta(t)+\omega^{2/3}}}.
\label{eq:directional-error-A}
\end{align}
Indeed, the first estimate follows from the energy bound. Substitution
of that estimate into the integrating-factor formula for \(e_A\) gives
\[
|e_A(t)|\leq C\eta
e^{-c\int_0^t\sqrt{\Delta(r)}\,dr}
+C\varepsilon\eta
\int_0^t e^{-c\int_s^t m_\eta(r)\,dr}\,ds .
\]
Since \(-\dot\Delta\asymp\varepsilon\), one has
\(|\dot m_\eta|\leq C\varepsilon/m_\eta\). On
\(\Delta\geq L\omega^{2/3}\), choosing \(L\geq L_0\) makes
\(\varepsilon/m_\eta^3\leq C L^{-3/2}\) small. Splitting the last
integral at the time when \(m_\eta(s)=2m_\eta(t)\), or equivalently
integrating by parts, gives the uniform kernel bound
\[
\int_0^t e^{-c\int_s^t m_\eta(r)\,dr}\,ds
\leq\frac{C}{m_\eta(t)}
\leq\frac{C}{\sqrt{\Delta(t)+\omega^{2/3}}},
\]
which proves \eqref{eq:directional-error-A}. The derived bounds imply
\(e_A=o(\sqrt\Delta)\) and \(e_B=o(\Delta)\) throughout the tube under
\eqref{eq:directional-scale-condition}; a standard continuation argument
therefore closes the bootstrap. This also proves that all constants in
\eqref{eq:directional-error-B}--\eqref{eq:directional-error-A} are
independent of \(\varepsilon,\eta\), of the initial point in the stated
compact incoming section, and of \(L\geq L_0\) while
\(L\omega^{2/3}\leq\Delta_0\).

At an incoming Airy section
\(\Delta=L\omega^{2/3}\), the coordinate transformation
\eqref{eq:explicit-normal-form-coordinates} and the preceding bounds
give
\begin{equation}
\frac{|a^\eta-\bar a|}{\omega^{1/3}}
+
\frac{|\mu^\eta-\bar\mu|}{\omega^{2/3}}
\leq
C\frac{\varepsilon\eta}{\omega^{2/3}}+o(1)
=o(1),
\label{eq:directional-inner-entry}
\end{equation}
The exponentially decaying term in
\eqref{eq:directional-error-A} is \(o(\omega^N)\) for every fixed
\(N>0\), because the drift bounds imply
\(\int_0^{t_L}\sqrt{\Delta(r)}\,dr\geq c/\varepsilon\). Equation
\eqref{eq:directional-scale-condition} now proves the final equality in
\eqref{eq:directional-inner-entry}. Continuous dependence from this
incoming section to the scaled crossing section propagates
\(|\mu^\eta-\bar\mu|=O(\varepsilon\eta)\). Hence
\(|t_c^\eta-t_c|\leq C|\mu^\eta-\bar\mu|/\omega=O(\eta)\), so the
crossing-time shift is \(o(\omega^{-1/3})\). The same
continuous-dependence argument shows that the full and reduced
trajectories converge to the same Airy tracking solution on every
compact subset of \(( -\infty,z_{\mathrm{Ai}})\). The exit-section
comparison used in Theorem~\ref{thm:adaptive-delay} then gives
\eqref{eq:full-dynamic-directional-time} and
\eqref{eq:full-dynamic-directional-mu}.
\end{proof}

\begin{remark}
The directional calculation is essential. Equation
\eqref{eq:euclidean-rate-obstruction} shows that a uniform Euclidean
all-direction growth bound is incompatible with a transverse adaptive
fold when $\partial_B\widehat F\neq0$. In contrast,
\eqref{eq:directional-error-B}--\eqref{eq:directional-inner-entry}
use the fact that the occupancy boundary layer forces only the fast
activity component; its projection onto the slow adaptive direction is
only $O(\varepsilon\eta)$.
\end{remark}

The two asymptotic regimes suggest the dimensionless ratio
\begin{equation}
\chi
=
\frac{\varepsilon\gamma_c}{\delta^{3/2}}.
\label{eq:ghost-crossover}
\end{equation}
When $\delta$ is interpreted as the distance from the frozen fold
along a common family of entry trajectories, $\chi$ provides an
order-of-magnitude indicator of the crossover between static
ghost escape and adaptive drift. It is not, by itself, a uniform
transition law for $\chi=O(1)$.

When $\chi\ll1$, adaptation changes negligibly during the
bottleneck passage and
$T_{\delta,\varepsilon}\sim\pi/\sqrt\delta$. When
$\chi=O(1)$, static escape and adaptive drift act on comparable
scales. When $\chi\gg1$ and the trajectory has tracked the
attracting branch before crossing the fold, the dynamic scale
$O((\varepsilon\gamma_c)^{-1/3})$ becomes relevant.

If $\gamma_c=0$, the adaptive vector field does not cross the
frozen fold transversely at leading order. This includes a fold
point lying on the adaptive nullcline $B=\Psi(A,C)$. The Airy
scaling then does not apply directly; the behaviour is determined
by higher-order drift terms or by the coupled-equilibrium
analysis of the preceding subsection.

Accordingly, the residence time entering the viability model
should be written as $T(C,\varepsilon)$. The frozen approximation
$T(C,\varepsilon)\sim\pi/\sqrt{\delta(C)}$ applies when
$\varepsilon\gamma_c\ll\delta(C)^{3/2}$. Near
$\delta(C)=O((\varepsilon\gamma_c)^{2/3})$, the adaptive drift
must be retained.

\section{From signalling activity to viability}
\label{sec:viability}
The preceding sections describe how treatment modifies signalling activity
through fast receptor occupancy, frozen-fold ghost dynamics, and slow
adaptive passage. In particular, Section~\ref{sec:coupled-fold} shows that
the resulting residence time should generally be regarded as
\(T(C,\varepsilon)\), with different leading behaviour depending on the
relative sizes of the distance from the fold and the adaptive timescale.
We now translate these signalling dynamics into a phenotypic response.
Because cell viability reflects the accumulated effect of signalling over
an observation interval, rather than receptor occupancy or instantaneous
activity alone, transient residence near the fold can remain visible in
viability measurements even after the occupancy variables have relaxed.

Let $N(t;C)$ denote the viable cell mass under inhibitor concentration $C$. We assume
\begin{equation}
\frac{dN}{dt}=G(A(t;C))N,
\label{eq:N}
\end{equation}
where $G\in C^2([0,1])$ is increasing, reflecting the interpretation that larger signalling activity
supports survival or proliferation. Let the untreated control satisfy
\begin{equation}
\frac{dN_{\mathrm{ctrl}}}{dt}=G(A_{\mathrm{ctrl}}(t))N_{\mathrm{ctrl}}.
\label{eq:Nctrl}
\end{equation}
Define viability relative to control by
\begin{equation}
V(t;C)=\frac{N(t;C)}{N_{\mathrm{ctrl}}(t)}.
\label{eq:Vdef}
\end{equation}

Assume that the treated and control populations have the same initial
normalization, \(N(0;C)=N_{\mathrm{ctrl}}(0)>0\),
so that \(V(0;C)=1\). From \eqref{eq:N}, the corresponding control
equation, and \eqref{eq:Vdef}, we obtain
\(\frac{d}{dt}\log V(t;C)=G\bigl(A(t;C)\bigr)
-G\bigl(A_{\mathrm{ctrl}}(t)\bigr)\).
Integrating from \(0\) to \(t\) gives
\begin{equation}
\log V(t;C)
=
\int_0^t
\left[
G\bigl(A(s;C)\bigr)
-
G\bigl(A_{\mathrm{ctrl}}(s)\bigr)
\right]\,ds.
\label{eq:Vintegral}
\end{equation}
Equivalently,
\[
V(t;C)
=
\exp\!\left\{
\int_0^t
\left[
G\bigl(A(s;C)\bigr)
-
G\bigl(A_{\mathrm{ctrl}}(s)\bigr)
\right]\,ds
\right\}.
\]

Equation~\eqref{eq:Vintegral} shows that viability is controlled by the
time-integrated activity gap, rather than by receptor occupancy alone.
This is where ghost-mediated delay and adaptive overshoot enter the
phenotype.

\subsection{Mechanism for the shoulder}

Suppose first that the treated dynamics are governed near threshold by the ghost equation
\eqref{eq:ghost} with concentration-dependent distance $\delta(C)>0$, where $\delta'(C)>0$. Let the
initial condition be fixed at $a(0)=a_0>0$, and define an effective escape time $T(C)$ as the time
required to leave a neighbourhood of the ghost region.

The following result is stated in the frozen-adaptation regime
\(\varepsilon\ll\delta(C)^{3/2}\),
in which the slow drift of \(B\) is asymptotically negligible during the
ghost passage. By Section~\ref{sec:coupled-fold}, the coupled residence
time \(T(C,\varepsilon)\) is then represented to leading order by the
frozen residence time \(T_{\delta(C)}\). For notational simplicity, we
write \(T(C)=T_{\delta(C)}\) in this regime.

\begin{theorem}[Ghost-induced shoulder]
Let \(T(C)=T_{\delta(C)}\) denote the residence time for passage from
\(a=M\) to \(a=-M\), as defined in Proposition~\ref{prop:ghost-residence-time}. Then \(T(C)\) is strictly
decreasing in \(C\), and
\[
T(C)\sim\frac{\pi}{\sqrt{\delta(C)}}
\qquad
\text{as }\delta(C)\downarrow0.
\]

Fix an observation time \(t>0\). Assume that trajectories which remain in
the ghost region have approximately the same growth rate as the untreated
control: for some \(r\ge0\),
\[
\left|
G(A(s;C))-G(A_{\mathrm{ctrl}}(s))
\right|
\le r,
\qquad
0\le s\le \min\{t,T(C)\}.
\]
If \(T(C)\ge t\), then
\[
e^{-rt}\le V(t;C)\le e^{rt}.
\]
Hence \(V(t;C)\approx1\) whenever \(rt\ll1\).

Assume further that, after escape from the ghost region, the treated
growth rate lies uniformly below the control growth rate: for some
\(g_0>0\),
\[
G(A(s;C))-G(A_{\mathrm{ctrl}}(s))
\le -g_0,
\qquad
T(C)\le s\le t.
\]
If \(T(C)<t\), then
\[
V(t;C)
\le
\exp\!\left[
rT(C)-g_0\bigl(t-T(C)\bigr)
\right].
\]

Consequently, if \(T(C)\ge t\) and \(rt\ll1\), viability remains close
to its control value. If \(T(C)<t\) and $g_0\bigl(t-T(C)\bigr)-rT(C)\gg1$,
then the displayed upper bound is exponentially small and viability is
substantially reduced. In particular, this occurs when \(T(C)\ll t\),
\(g_0t\gg1\), and \(rT(C)\) is negligible. Because \(T(C)\) decreases
with concentration, the transition occurs over the concentration range
in which \(T(C)\approx t\), producing a dose--response shoulder.
\end{theorem}

\begin{proof}
The monotonicity and asymptotic formula follow from the explicit expression
for \(T(C)\) and the assumption \(\delta'(C)>0\).

If \(T(C)\ge t\), then \eqref{eq:Vintegral} and the bound on the growth-rate
gap give \(-rt\le\log V(t;C)\le rt\),
which proves the first estimate.

If \(T(C)<t\), split the viability integral at the escape time:
\[
\log V(t;C)
=
\int_0^{T(C)}
\left[
G(A(s;C))-G(A_{\mathrm{ctrl}}(s))
\right]ds
+
\int_{T(C)}^t
\left[
G(A(s;C))-G(A_{\mathrm{ctrl}}(s))
\right]ds.
\]
The two assumptions give
\(\log V(t;C)\le rT(C)-g_0\bigl(t-T(C)\bigr)\),
and exponentiation gives the stated bound.
\end{proof}

The shoulder is therefore a consequence of time-scale separation and critical slowing in the
downstream activity system, not a consequence of a static receptor-occupancy curve. The receptor
occupancy state may already have relaxed to $X_{\mathrm{qs}}(C)$, while the signalling activity
$A(t)$ remains transiently elevated because of feedback and ghost dynamics.

\subsection{A proof-of-concept fit to experimental viability data}
\label{sec:Lkyn-fit}

To illustrate how the proposed ghost mechanism may be confronted with
experimental data, we consider the L-kynurenine viability measurements
reported in Fig.~6 of \cite{Anguelov2023}. Melanoma cells were exposed to
L-kynurenine at concentrations $C=1,2,3,4$ mM, and cell viability was
measured at observation times $t=24,48,$ and $72$ h using the crystal
violet assay. Several replicate measurements are available at each
concentration--time pair, and all reported replicate values were used in
the parameter estimation rather than first averaging the data.

The complete model developed in the previous sections contains the hidden
occupancy, signalling, and adaptation variables $X$, $A$, and $B$. These
variables were not measured in the available experiment, so the full
activity--adaptation system cannot be identified directly from viability
measurements alone. Instead, we construct a reduced observation model that
retains the principal prediction of the ghost analysis, namely that the
residence time near the vanished high-activity equilibrium scales according
to the inverse square root of the distance from the saddle-node.

As described in Section~\ref{sec:ghost}, the experimental concentration $C$
acts on the activity dynamics through the effective inhibition
\(I(C)=H\bigl(X_{\mathrm{qs}}(C)\bigr)\). Near the saddle-node, the
normal-form parameter may be written locally as
\(\mu=\alpha(B-B_c)-\beta(I-I_c)+\text{higher-order terms}\), with
\(\beta>0\).
The ghost residence time therefore scales as
\begin{equation}
T_{\mathrm g}
\sim
\frac{\pi}{\sqrt{-\mu}},
\qquad \mu<0.
\label{eq:ghost-fit-general}
\end{equation}

Since neither \(I(C)\) nor \(B(t)\) is measured in the present experiment,
the distance from the fold cannot be determined directly from the viability
data. To obtain a reduced proof-of-concept observation model, we assume that
\(B\) is approximately constant during passage through the ghost region and
that concentration unfolds the saddle-node transversely.

Let \(C_c\) denote the critical concentration at which the saddle-node
occurs. A Taylor expansion of the unfolding parameter about \(C=C_c\) gives
\begin{equation}
\mu(C)
=
\mu'(C_c)(C-C_c)
+
O\bigl((C-C_c)^2\bigr).
\label{eq:mu-concentration}
\end{equation}
On the ghost side of the fold, choose the sign convention such that
\[
-\mu(C)
=
s(C-C_c)
+
O\bigl((C-C_c)^2\bigr),
\qquad s>0,
\qquad C>C_c.
\]
Combining this relation with the inverse-square-root saddle-node scaling
gives the local approximation
\begin{equation}
T_{\mathrm g}(C)
=
\frac{\tau}{\sqrt{C-C_c}},
\qquad C>C_c,
\label{eq:fit-delay-general}
\end{equation}
where \(\tau>0\) absorbs the normal-form scaling constants.

For the initial four-parameter fit, we adopt the parsimonious working
approximation \(C_c=0\), giving
\begin{equation}
T_{\mathrm g}(C)
=
\frac{\tau}{\sqrt{C}},
\qquad C>0.
\label{eq:fit-delay}
\end{equation}
This choice reduces the dimension of the observation model. It is a
modelling convention rather than evidence that the physical saddle-node
occurs at zero inhibitor concentration; its adequacy is assessed below by
profiling \(C_c\). The untreated control \(C=0\) is used for normalization
and is not evaluated through \eqref{eq:fit-delay}.

The viability is then modelled by
\begin{equation}
V_{\mathrm{fit}}(t,C)
=
v_\infty
+
(1-v_\infty)
\exp\!\left[
-kC\bigl(t-T_{\mathrm g}(C)\bigr)_+^{\,p}
\right],
\qquad
T_{\mathrm g}(C)=\frac{\tau}{\sqrt{C}},
\qquad C>0,
\label{eq:fit-model}
\end{equation}
where \((x)_+=\max\{x,0\}\). Here $v_\infty$ denotes the residual
viability after prolonged exposure, $k$ determines the rate of viability
loss after escape from the ghost region, and the exponent $p$ allows the
post-delay decay to deviate from a pure exponential.
Equation~\eqref{eq:fit-model} should therefore be viewed as a reduced
observation model rather than as a replacement for the mechanistic
activity--adaptation system.

The parameter vector \(\theta=(v_\infty,k,p,\tau)\) contains the
dimensionless parameters \(v_\infty\) and \(p\). Since
concentration is measured in mM and time in hours, the remaining
parameter units are
\([k]=\mathrm{mM}^{-1}\mathrm{h}^{-p}\) and
\([\tau]=\mathrm{h}\sqrt{\mathrm{mM}}\).
The parameter vector \(\theta\) was estimated by nonlinear least squares.
If \(V_{ijr}\) denotes replicate \(r\) measured at concentration \(C_i\)
and observation time \(t_j\), the parameters were obtained by minimizing
\begin{equation}
\mathcal{J}(\theta)
=
\sum_{i,j,r}
\left[
V_{ijr}
-
V_{\mathrm{fit}}(t_j,C_i;\theta)
\right]^2.
\label{eq:fit-objective}
\end{equation}

The optimization was carried out using MATLAB's \texttt{fminsearch}.
Since the objective function is nonlinear, several different initial
guesses were used and the parameter vector producing the smallest value
of \eqref{eq:fit-objective} was retained. To enforce biologically
admissible parameter values, the optimization was performed in transformed
coordinates, ensuring \(0<v_\infty<0.4\), \(k>0\),
\(0.5<p<3.5\), and \(\tau>0\).

Parameter uncertainty was assessed using a concentration--time-stratified
nonparametric bootstrap. Within each of the 12 concentration--time groups,
the replicate viability measurements were sampled with replacement while
preserving the original number of observations in that group. The
four-parameter model was then refitted to each of 2000 bootstrap samples.
The 2.5th and 97.5th percentiles of the resulting parameter estimates were
used as percentile 95\% bootstrap intervals. Sensitivity to individual
observations was assessed separately by a leave-one-observation-out
jackknife.

The identifiability of the effective critical concentration was examined
by profiling \(C_c\) over \(0\le C_c<1~\mathrm{mM}\).
At each fixed value of \(C_c\), the observation model
\eqref{eq:fit-model} was evaluated using the shifted delay $T_{\mathrm g}(C)
=
\frac{\tau}{\sqrt{C-C_c}},
\qquad C>C_c$, and the remaining four parameters
\((v_\infty,k,p,\tau)\) were re-estimated. Let
\(\operatorname{SSE}(C_c)\) denote the resulting residual sum of squares. 
For purposes of the profile calculation, we adopt a working error model
in which the replicate residuals are independent Gaussian variables with
a common unknown variance. After profiling out this variance, the profile
likelihood-ratio statistic is
\[
D(C_c)
=
n\log\left(
\frac{\operatorname{SSE}(C_c)}
     {\operatorname{SSE}_{\min}}
\right),
\]
where \(n=40\) is the number of replicate observations.

Because \(C_c\) is constrained to be nonnegative, a profile optimum at
the boundary \(C_c=0\) requires a one-sided 95\% upper profile bound. The corresponding cutoff is
\(\chi^2_{1,0.90}=2.706\), equivalently the 95th percentile of the
boundary distribution
\(\frac12\chi_0^2+\frac12\chi_1^2\) \cite{SelfLiang1987}. The
approximate one-sided 95\% profile set was therefore defined by
\(D(C_c)\le2.706\).

 The fixed and estimated-\(C_c\) models were also compared using
\[
\mathrm{AIC}_c
=
n\log\left(\frac{\operatorname{SSE}}{n}\right)
+2K+\frac{2K(K+1)}{n-K-1},
\]
where \(K\) is the total number of estimated parameters, including the
residual variance. Thus \(K=5\) for the fixed-\(C_c\) model and \(K=6\)
for the estimated-\(C_c\) model. The profile was evaluated
at 201 equally spaced values of \(C_c\) over the stated interval, with
linear interpolation used to locate the upper profile crossing. To
reproduce the numerical \(\mathrm{AIC}_c\) values reported below, the
residual sum of squares in this expression was evaluated with viability
expressed in percentage points. This common rescaling does not affect the
\(\mathrm{AIC}_c\) difference between the two models.

The resulting parameter estimates were
\begin{equation}
v_\infty=0.0673,\qquad
k=8.10\times10^{-4}~\mathrm{mM}^{-1}\mathrm{h}^{-p},\qquad
p=1.743,\qquad
\tau=20.425~\mathrm{h}\sqrt{\mathrm{mM}}.
\label{eq:fit-values}
\end{equation}
Under \eqref{eq:fit-delay}, these values give effective ghost residence
times \(T_{\mathrm g}(1)=20.4\), \(T_{\mathrm g}(2)=14.4\),
\(T_{\mathrm g}(3)=11.8\), and \(T_{\mathrm g}(4)=10.2\) h.

The fit yields $R^2=0.9567$ with a root-mean-square error of approximately
$6.75$ percentage points. 

\begin{figure}[htbp]
\centering
\includegraphics[width=0.82\textwidth]{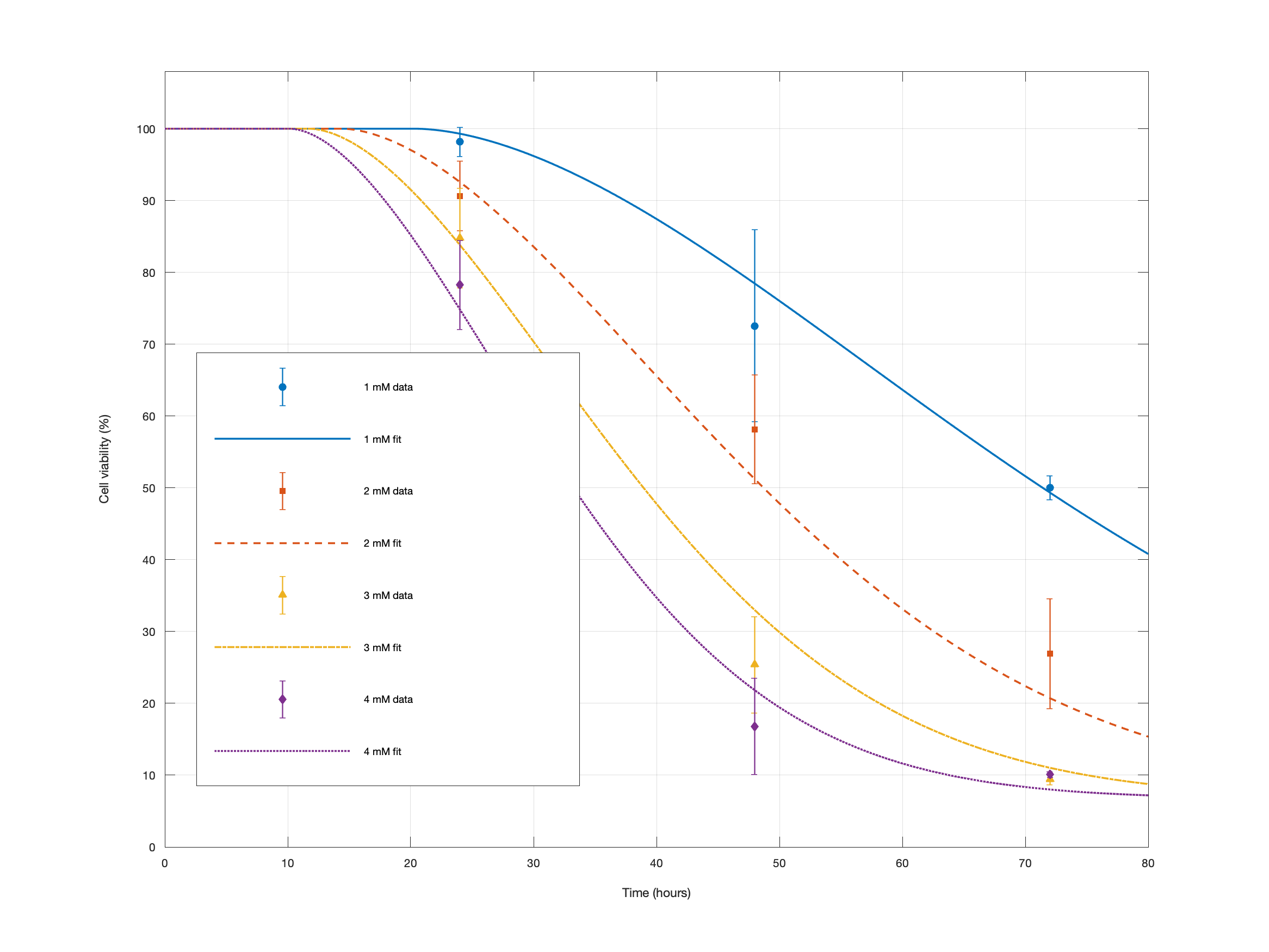}
\caption{Proof-of-concept fit to the L-kynurenine viability data reported
in \cite{Anguelov2023} using the reduced ghost-inspired observation model.
Symbols denote the experimental means with standard deviations, while the
solid curves are obtained from \eqref{eq:fit-model}. For this fit, the
effective distance from the saddle-node is assumed to vary proportionally
with inhibitor concentration, giving
$T_{\mathrm g}(C)=\tau/\sqrt{C}$.}
\label{fig:Lkyn-fit}
\end{figure}

The concentration--time-stratified bootstrap gave the 95\% intervals
\(v_\infty\in[0.0467,0.0869]\),
\(k\in[2.23\times10^{-4},2.69\times10^{-3}]
~\mathrm{mM}^{-1}\mathrm{h}^{-p}\),
\(p\in[1.466,2.033]\), and
\(\tau\in[12.25,27.89]~\mathrm{h}\sqrt{\mathrm{mM}}\).
The corresponding bootstrap coefficients of variation were \(15.4\%\),
\(65.9\%\), \(8.4\%\), and \(20.4\%\), respectively. Thus \(p\) was
comparatively stable, whereas \(k\) was the least precisely determined
individual parameter. The bootstrap estimates also exhibited strong
parameter trade-offs:
\(\operatorname{corr}(k,p)=-0.918\),
\(\operatorname{corr}(k,\tau)=0.817\), and
\(\operatorname{corr}(p,\tau)=-0.832\).
Consequently, the fitted viability curves and the qualitative
concentration dependence of the delay are more robust than the individual
estimate of \(k\). In the leave-one-observation-out analysis, the maximum
absolute changes in \(v_\infty\), \(k\), \(p\), and \(\tau\) were
\(11.6\%\), \(72.8\%\), \(6.9\%\), and \(19.7\%\), respectively, which
supports the same conclusion.

When \(C_c\) was allowed to vary, the profile optimum occurred at the
boundary \(C_c=0\). Estimating \(C_c\) produced no material reduction
in the residual sum of squares and increased the corrected Akaike
information criterion from \(164.47\) to \(167.25\), giving
\(\Delta\mathrm{AIC}_c=2.78\). The approximate one-sided 95\% profile
set was
\[
0\leq C_c\leq 0.47~\mathrm{mM}.
\]
Thus the available data do not identify a positive critical concentration,
and the profile and model-comparison results support retaining the working
approximation \(T_{\mathrm g}(C)=\tau/\sqrt{C}\).
Figure~\ref{fig:Lkyn-fit} compares the resulting fitted curves with the
experimental means and standard deviations.

The fitted curves reproduce the principal concentration-dependent timing of
the viability decline: increasing inhibitor concentration is associated,
under the approximation \eqref{eq:fit-delay}, with a shorter effective
residence time and an earlier loss of viability. This agreement should not,
however, be interpreted as direct evidence for a saddle-node bifurcation in
the underlying biological system. Since only viability was measured, the
hidden occupancy, signalling, and adaptive variables cannot be identified
uniquely, and alternative delayed-response mechanisms may generate similar
temporal behaviour.

In particular, the present data do not determine the dose--occupancy map
$X_{\mathrm{qs}}(C)$, the effective inhibition function $I(C)$, the
physical critical concentration \(C_c\), or the distance $\mu$ from the
saddle-node independently. A more stringent test of the proposed mechanism
would require time-resolved measurements of receptor occupancy or
intracellular signalling together with viability, allowing the relation
between concentration, effective inhibition, and ghost residence time to
be tested directly.

\subsection{Mechanism for the bump: adaptive compensation}

A purely inhibitory ghost model cannot produce $V(t;C)>1$ relative to control, because that would
require the treated trajectory to have a larger time-integrated net growth rate than the untreated one.
To account for early low-dose overshoot, we use the adaptive variable $B$ rather than an ad hoc forcing
term.  The simplest reduced activity--adaptation subsystem is
\begin{align}
\dot a &= \nu(C)-\rho b-a^2,
\label{eq:abump}\\
\dot b &= \varepsilon\bigl(\kappa(C)+\xi a-b\bigr),
\qquad 0<\varepsilon\ll 1,
\label{eq:bbump}
\end{align}
where \(a=A-A_c\) is the local activity coordinate,
\(\nu(C)\) captures the direct occupancy-mediated drive,
\(\kappa(C)\) is a treatment-dependent adaptive input, and
\(\rho,\xi>0\).

For fixed \(a\) and \(C\), the adaptive variable \(b\) relaxes toward
the quasi-steady value \(b_{\mathrm{qs}}(a,C)=\kappa(C)+\xi a\).
Because \(0<\varepsilon\ll1\), \(b\) cannot immediately attain this
value following treatment. The negative-feedback term \(-\rho b\) may
therefore differ transiently from its quasi-steady value. For suitable
treatment dependence of \(\nu\) and \(\kappa\), and suitable initial
conditions, this adaptive lag can allow the treated activity to exceed
the untreated activity over a finite time interval. If the resulting
time-integrated activity gap is positive, Equation~\eqref{eq:Vintegral}
shows that \(V(t;C)>1\) may occur. Thus adaptive lag provides a possible,
but not universal, mechanism for an early-time viability bump.

This mechanism is more defensible than imposing a scalar
concentration-dependent bump term by hand because the overshoot arises
from a dynamical adaptive state. It is also consistent with the biological
logic of adaptive rewiring and feedback relief in targeted therapy
\cite{Chandarlapaty2012,Rosell2015,Pazarentzos2015}.

\subsection{$IC_{50}(t)$ and transient potency shifts}

Whenever \(C\mapsto V(t;C)\) is continuous and strictly decreasing with
a crossing at \(1/2\), define the time-dependent half-maximal inhibitory
concentration by \(V\bigl(t;IC_{50}(t)\bigr)=1/2\).
Because viability depends on the integrated activity history, not just on receptor occupancy,
$IC_{50}(t)$ need not be constant in time. In the ghost regime, intermediate concentrations preserve
activity long enough that the early dose--response curve appears flatter than the late one. As the ghost
is exited and adaptive compensation relaxes, the apparent potency increases, and $IC_{50}(t)$ shifts
downward.

The model therefore predicts that two assays with the same binding kinetics can report different
potencies if they are sampled at different times or if they differ in trafficking and feedback context.
This is a central conceptual advantage of the network-embedded formulation.

\section{Discussion}
\label{sec:discussion} 
The manuscript separates three levels of description: fast receptor occupancy, intracellular
signalling activity, and phenotypic viability. The occupancy layer is not specified in full
mechanistic detail. Instead, it is represented by an abstract variable \(X(t)\) that relaxes rapidly
to a quasi-steady state \(X_{\mathrm{qs}}(C)\). This captures the essential role of competitive
occupancy models such as those in \cite{Anguelov2023} while avoiding unnecessary assumptions about
receptor trafficking or feedback from adaptation to occupancy.

This reformulation clarifies the role of the slow adaptive variable \(B\). In the present model,
\(B\) acts at the signalling-network level; it does not directly alter receptor occupancy. Therefore
the quasi-steady occupancy state is written as \(X_{\mathrm{qs}}(C)\) rather than
\(X_{\mathrm{qs}}(C,B)\). A dependence on \(B\) would require an explicit biological mechanism,
such as receptor downregulation, altered trafficking, or changes in binding affinity. Since those
mechanisms are not included, the cleaner formulation avoids this hidden coupling.

The ghost mechanism arises downstream in the activity--adaptation subsystem
\cite{Stanoev2020,Aldridge2006}. Occupancy provides a fast treatment-dependent input, while the
nonlinear interaction of signalling activity \(A\) and slow adaptation \(B\) produces saddle-node
criticality, ghost residence, delayed pathway shutdown, and transient apparent resistance. Thus the
extension preserves the mathematical spirit of competitive models such as in
\cite{Anguelov2023} while moving the source of nontrivial transient dynamics to the biologically
appropriate signalling layer.

For the sigmoidal activity map, this mechanism is no longer merely a normal-form hypothesis.
Positive feedback satisfying \(\gamma>4\) produces distinct fold candidates, and whenever
\(B_{\mathrm f}(C)\in(0,1)\), one is a biologically admissible high-activity frozen fold. At a
selected concentration \(C_{\mathrm f}\), the inequality
\(\Psi(A_{\mathrm f},C_{\mathrm f})>B_{\mathrm f}(C_{\mathrm f})\) directs the adaptive flow
across that fold, with the normalized transverse coefficient \(\gamma_c\) given by
\eqref{eq:explicit-gamma-c}. The directional transfer argument in
Subsection~\ref{sec:directional-transfer} avoids imposing uniform Euclidean contraction at the
fold, which is incompatible with the nonzero \(B\)-to-\(A\) coupling required for the adaptive
crossing.

The fast--slow formulation also distinguishes two related but mathematically different notions
of criticality. The fold curve of the frozen activity subsystem is a fold of the critical manifold
obtained when \(B\) is treated as a parameter. A saddle-node equilibrium of the coupled
activity--adaptation system must additionally lie on the adaptive nullcline
\(B=\Psi(A,C)\), and its degeneracy condition includes the feedback contribution
\(\partial_B\widehat F\,\partial_A\Psi\). Thus a frozen activity fold should not in general be
identified with a saddle-node equilibrium of the coupled system. Nevertheless, it remains
dynamically relevant because the slowly evolving adaptive variable may carry a trajectory through
its neighbourhood.

This distinction produces two asymptotic delay regimes, with the crossover ratio
\(\chi=\varepsilon\gamma_c/\delta^{3/2}\). When
\(\varepsilon/\delta^{3/2}\to0\), or equivalently \(\chi\to0\) for fixed positive
\(\gamma_c\), adaptation is negligible during the bottleneck and the frozen-fold law
\(T\sim\pi/\sqrt{\delta}\) remains valid. When \(\chi=O(1)\), static escape and adaptive drift
act on comparable scales, although \(\chi\) is not a uniform transition law. For a trajectory that
has tracked the attracting branch before a transverse crossing, the dynamic delay satisfies
\(t_{\mathrm{exit}}-t_c\sim z_{\mathrm{Ai}}(\varepsilon\gamma_c)^{-1/3}\), with exit
displacement of order \((\varepsilon\gamma_c)^{2/3}\). If \(\gamma_c=0\), this Airy scaling
does not apply directly. The transfer of these reduced laws to the full occupancy system also
requires \(\eta=o(\delta)\) in the frozen regime and
\(\eta=o((\varepsilon\gamma_c)^{2/3})\) in the dynamic regime, as established in
Subsection~\ref{sec:directional-transfer}.

The present viability measurements do not distinguish these regimes. The proof-of-concept
observation model assumes approximately frozen adaptation and uses
\(T_{\mathrm g}(C)=\tau/\sqrt{C}\); it therefore tests compatibility with a
concentration-dependent delay, not the dynamic Airy law. Because only viability was observed,
the data cannot independently estimate \(\varepsilon\), \(\gamma_c\), or the distance
\(\delta\) from the fold. The fit is therefore compatible with a concentration-dependent delay
phenotype, but it does not establish a frozen-fold passage, a coupled saddle-node, or a particular
adaptive timescale. Similarly, Figure~1 illustrates the theoretical bifurcation geometry rather
than reconstructing it from the limited data. Distinguishing these mechanisms would require
time-resolved measurements of signalling activity and adaptive regulation together with viability.

Several extensions are natural. One may estimate the occupancy parameters from binding data and
the activity/adaptation parameters from time-resolved phospho-signalling measurements. One may
also include stochastic forcing near the fold to study variability in ghost residence times, or
multiple adaptive nodes to distinguish phosphatase induction from transcriptional rewiring. Those
steps would move the theory closer to direct assay fitting while preserving the conceptual core
developed here.

\noindent\textbf{Data Availability Statement. } 
The data analyzed in this study were part of experimental measurements originally published in \cite{Anguelov2023} and are graphically represented in Figure 2

\noindent\textbf{Funding Information. } The second author was supported by the DST/NRF SARChI Chair on Mathematical Models and Methods in Bioengineering and Biosciences at the University of Pretoria.

\bibliographystyle{wmaainf}
\bibliography{IMA_extension.bib}

\end{document}